\documentclass[a4paper, 10pt]{article}
\usepackage[colorlinks = true,
            linkcolor = black,
            urlcolor  = blue,
            citecolor = black]{hyperref}
\usepackage{amsmath, amsthm, amssymb, amsfonts, amscd}
\usepackage[left=2.5cm,right=2.5cm,top=3cm,bottom=2.5cm,a4paper]{geometry}
\usepackage{mathtools}
\usepackage{thmtools}
\usepackage{scalerel}
\usepackage{graphicx}
\usepackage[titletoc]{appendix}
\usepackage{cleveref}
\DeclareGraphicsExtensions{.pdf,.png,.jpg}

\DeclareMathOperator{\coker}{coker}
\DeclareMathOperator{\Aut}{Aut}
\newcommand{\inv}{{\operatorname{inv}}}
\newcommand{\Gal}{\operatorname{Gal}}
\newcommand{\Hom}{\operatorname{Hom}}
\newcommand{\GL}{\operatorname{GL}}
\newcommand{\SL}{\operatorname{SL}}
\newtheorem{theorem}{Theorem}[section]
\newtheorem{proposition}[theorem]{Proposition}
\newtheorem{corollary}[theorem]{Corollary}
\newtheorem{lemma}[theorem]{Lemma}

\newtheorem{remark}[theorem]{Remark}
\newtheorem{definition}[theorem]{Definition}
\newtheorem{example}[theorem]{Example}
\newtheorem{assumption}[theorem]{Assumption}
\usepackage{titling}
\usepackage{url}
\usepackage[nosort]{cite}
\newcommand{\changeurlcolor}[1]{\hypersetup{urlcolor=#1}} 

\usepackage{enumitem}
\usepackage{tikz-cd}

\usepackage{array}
\makeatletter
\newcommand{\thickhline}{%
    \noalign {\ifnum 0=`}\fi \hrule height 1pt
    \futurelet \reserved@a \@xhline
}

\newcolumntype{"}{@{\hskip\tabcolsep\vrule width 1pt\hskip\tabcolsep}}
\makeatother

\newcommand{\Z}{\mathbb{Z}}
\newcommand{\Q}{\mathbb{Q}}
\newcommand{\Spec}{\operatorname{Spec}}

\title{\large{\textbf{ARITHMETIC CHERN-SIMONS THEORY WITH REAL PLACES}}}
\author{\normalsize{JUNGIN LEE, JEEHOON PARK}}
\date{}

\usepackage{fancyhdr}

\newcommand\shorttitle{ARITHMETIC CHERN-SIMONS THEORY WITH REAL PLACES}
\newcommand\authors{JUNGIN LEE, JEEHOON PARK}

\usepackage{sectsty}
\sectionfont{\large}
\subsectionfont{\normalsize}

\begin{document}
\maketitle

\vspace{-10mm}

\begin{abstract}
The goal of this paper is two-fold: we generalize the arithmetic Chern-Simons
theory over totally imaginary number fields studied in \cite{ACST1, ACST} to arbitrary number fields (with real places) and provide
new examples of non-trivial arithmetic Chern-Simons invariant with coefficient 
$\Z/n\Z$ $ (n \geq 2)$ associated to a non-abelian gauge group.
The main idea for the generalization is to use cohomology with compact support (see \cite{ADT}) to deal with real places. Before the results of this paper, non-trivial examples were limited to some non-abelian gauge group with coefficient $\Z/2\Z$ in \cite{ACST} and the abelian cyclic gauge group with coefficient $\Z/n\Z$ in \cite{BCG}. 
Our non-trivial examples (with non-abelian gauge group and general coefficient $\Z/n\Z$) will be given by a simple twisting argument based on examples of \cite{BCG}.
\end{abstract}

\vspace{3mm}

{\small
}

\tableofcontents

\section{Introduction} \label{Sec1}

Arithmetic topology is an area which studies an analogy between knots in 3-manifold theory and primes in number theory, based on homotopical analogies between them. Such an analogy was first suggested in the 1960's by B. Mazur. Later, M. Kapranov \cite{Kap} and A. Reznikov \cite{Rez} examined the analogy further and proved some interesting results. It was M. Morishita who studied arithmetic topology more systematically in his series of papers including \cite{Mor1},\cite{Mor2}; he also published a wonderful book \cite{Mor3} on the subject.

Chern-Simons theory is a 3-dimensional topological quantum field theory with gauge group.
It has been used to understand knot invariants such as the Jones polynomial.
Then the analogy between knots and primes (arithmetic topology) leads to a question whether there is an arithmetic analogue of Chern Simon theory which can be used to understand number-theoretic invariants further.
It was Minhyong Kim (see \cite{ACST1}) who answered this question when the gauge group is finite (also suggested a way to define the theory when the gauge group is a $p$-adic analytic group).
He defined the arithmetic Chern-Simons action (also called the arithmetic Chern-Simons invariant) by applying the ideas of Dijkgraaf and Witten (see \cite{DW}) on a 3-dimensional topological quantum field theory to an arithmetic curve, namely the spectrum of the ring of integers in an algebraic number field. Then, in \cite{ACST}, H. Chung, D. Kim, M. Kim, H. Yoo together with the second-named author continued to study the theory to provide non-trivial examples of the arithmetic Chern-Simons functional. Also, an intimate connection between the abelian arithmetic Chern-Simons theory and arithmetic linking numbers of prime ideals (number-theoretic invariants) is investigated in \cite{ALN}.

Let us briefly recall the setting of \cite{ACST1, ACST}. Let $F$ be a totally imaginary number field and $\mathcal{O}_F$ the ring of integers of $F$.
Let $X=\Spec (\mathcal{O}_F)$ and $n \geq 2$ be a natural number. 
Assume that $F$ contains a primitive $n$-th roots of unity.
Then there is a canonical isomorphism
\begin{eqnarray} \label{inviso}
H^3(X, \Z/n\Z) \xrightarrow{\inv}  \frac{1}{n}\Z/\Z.
\end{eqnarray}
Denote $\pi^{un} := \Gal(F_{un}/F)$ where $F_{un}$ is the maximal unramified extension of $F$ in an algebraic closure $\overline{F}$. 
Let $A$ be a finite group (called the gauge group).
Consider the space of fields on the space time $X$:
$$
\mathcal{M}(A):=\Hom_{cont}(\pi^{un}, A)/A,
$$ 
where $A$ acts on the target by conjugation.
Fix a group cohomology class $c \in H^3(A, \Z/n\Z)$ where $A$ acts on $\Z/n\Z$ trivially. 
Define the \textbf{arithmetic classical Chern-Simons action (or invariant) without boundary} by the function
$$ CS_c : \mathcal{M}(A) \rightarrow \frac{1}{n}\Z/\Z, \quad [\rho] \mapsto \inv(j^3_{un}(\rho^*(c))). $$
where $j^3_{un}: H^3(\pi^{un}, \Z/n\Z) \to H^3(X, \Z/n\Z)$ is the edge map 
induced from the Hochschild-Serre spectral sequence (see (\ref{jun}) for an alternative precise definition of $j^3_{un}$).
Since $\rho^*(c)$ depends only on the class $[\rho]$, $CS_c$ is well-defined.

A technical reason why $F$ is assumed to be totally imaginary in \cite{ACST1, ACST}
is that the isomorphism (\ref{inviso}) does not hold anymore when $F$ has a real place.
A natural idea for generalizing the arithmetic Chern-Simons theory to an arbitrary
number field is to use cohomology with compact support, since there is a canonical isomorphism
(see Proposition \ref{prop24})
$$
\operatorname{inv}: H_c^3(X, \Z/n\Z) \simeq \frac{1}{n}\Z/\Z
$$ 
for any number field $F$ containing a primitive $n$-th roots of unity.
Based on this isomorphism, we extend the arithmetic classical Chern-Simons action without boundary to an arbitrary number field.
Since we assume that $F$ contains a primitive $n$-th roots of unity\footnote{If $n\geq3$ under this assumption, then $F$ is totally imaginary.}, the actual new case
compared to \cite{ACST} is that $n=2$ and $F$ is a real number field.
But we will state propositions and theorems for arbitrary number fields $F$ and natural numbers $n$ for efficiency of the presentation of proofs and examples.

Let $S_f$ be any finite set of finite primes of $F$, 
$X_{\infty}$ be the set of all real places of $F$, 
$S := S_f \cup X_{\infty}$, 
$\widetilde{S}$ be any set satisfying $S_f \subset \widetilde{S} \subset S$, 
$U_S := \Spec(\mathcal{O}_F\left [ \frac{1}{S_f} \right ])$ and
$\pi_{\widetilde{S}} := \Gal(F_{un}^{\widetilde{S}}/F)$ where $F_{un}^{\widetilde{S}}$ is the maximal extension of $F$ in $\overline{F}$ unramified outside $\widetilde{S}$. 
When working with the arithmetic classical Chern-Simons action with boundary (see (\ref{wb}))
and proving the decomposition formula\footnote{This provides a way of computing the arithmetic classical Chern-Simons action without boundary using a comparison between local unramified and global ramified trivializations of a Galois three cocycle.} (see Section \ref{Sec3}),
the vanishing of the \'etale cohomology group $H^3(U_{S_f}, \Z/n\Z)$ plays an important role.
Another key difference between totally imaginary fields and real number fields lies in the vanishing behavior of $H^3$. If $S_f$ contains places of $F$ above $n$, then $H^3(U_{S_f}, \Z/n\Z)=0$ for an arbitrary number field $F$ but $H^3(U_{\tilde S}, \Z/n\Z)\neq 0$ for a real field $F$ and $\tilde S$ which strictly includes $S_f$. 
Due to this difference, our Chern-Simons action and decomposition formula differ from those in \cite{ACST}.

In Section 5, \cite{ACST}, the authors used the decomposition formula to construct infinitely many totally imaginary number fields $F$ in which the action $CS_c([\rho])$ is non-zero for some $c$ and $\rho$ when $n=2$ and the gauge group $A$ is $\Z/2\Z, \Z/2\Z \times \Z/2\Z,$ or the symmetric group $S_4$. 
We provide analogues of those examples for the totally real case.

In Section 6, \cite{ACST}, they showed non-solvability of a certain case of the embedding problem (biquadratic fields are involved) based on non-vanishing (\cite[Theorem 6.2]{ACST}) of the arithmetic Chern-Simons action with gauge group $V_4$, the Klein four group. Unfortunately, such an arithmetic application is not available in our case with totally real fields.

An example of non-vanishing of $CS_c([\rho])$ for the case $n > 2$ but with the abelian cyclic gauge group $A=\Z/n\Z$ was first provided by the work of F. M. Bleher, T. Chinburg, R. Greenberg, M.  Kakde, G. Pappas and M. J. Taylor, in \cite{BCG}.
 When $A=\Z/n\Z$ for any $n\geq 2$, they found infinitely many totally imaginary number fields $F$ and representations
 $\rho$ such that $CS_c([\rho]) \neq 0$ for a choice of $c$ (see \cite[Theorem 1.2]{BCG}).
 Here we prove that there are infinitely many totally imaginary number fields $F$ such that $CS_c([\rho]) \neq 0$ for some $\rho$ and $c$, when the gauge group is $A=G \rtimes \Z/n \Z$ (the semidirect product of some finite group $G$ by $\Z/n \Z$) for any $n\geq 2$ (See Lemma \ref{lem4b} and Theorem \ref{thm4c}.) For instance, we will give examples of non-vanishing arithmetic Chern-Simons action
 (see Example \ref{ex4e}) where the gauge group is the general linear group $\GL(r, \mathbb{F}_p)$ over the finite field $\mathbb{F}_p$ ($r \geq 2$) and $n=p-1$.

 There exists the method of Zink \cite{ZINK} and Conrad-Masullo \cite{CM} of generalizing the \'etale cohomology of $X=\Spec(\mathcal{O}_F)$ for totally imaginary number fields $F$, which was studied by Mazur \cite{MAZ}, to arbitrary number fields. 
 This provides an alternative way of defining the arithmetic classical Chern-Simons action. In the appendix, we add such a viewpoint. Note that the approach based on cohomology with compact support has a benefit of providing a natural framework to prove the decomposition formula.

The methods of this paper using cohomology with compact support carry over \textit{mutatis mutandis} to the case of global function fields provided $n$ is prime to the characteristic of the field. So there is also a version of 
arithmetic Chern-Simons theory for a finite extension $F$ of the field of rational functions $\mathbb{F}_p(T)$ when $n$ is prime to a prime number $p$.

\vspace{1em}

We briefly explain the contents of each section. In Section \ref{Sub21}, we set up basic definitions
and notations. In Section \ref{Sub22} we review the theory of cohomology with compact support based on Milne's book \cite{ADT}. We define the arithmetic classical Chern-Simons action without boundary (in Section \ref{Sub23}) and with boundary (in Section \ref{Sub24}). Since we do not deal with the quantization of the theory and we will omit the word ``classical" in the body of the paper. Section \ref{Sec3} is devoted to a statement and a proof of the decomposition formula.

Section \ref{Sec4} is about examples. We provide analogous examples to Section 5, \cite{ACST} (in Section \ref{Sub41}) and to Section 6, \cite{ACST} (in Section \ref{Sub44}) 
in the case of totally real number fields with the coefficient $\Z/2\Z$ and the gauge group $A=\Z/2\Z, \Z/2\Z\times \Z/2\Z$, or $S_4$.
In Section \ref{Sub45}, we provide non-vanishing examples with general $n$ and some non-abelian gauge group based on the result of \cite{BCG}. Finally, we add an alternative viewpoint of the arithmetic Chern-Simons theory with real places based on the generalized \'etale cohomology theory developed by Zink and Conrad-Masullo.

\section{Arithmetic Chern-Simons action} \label{Sec2}

In this section, we define the arithmetic Chern-Simons action for an arbitrary number field. To deal with the real places of a number field $F$, we use cohomologies with compact support.

\subsection{Definitions and notations} \label{Sub21}

\noindent Let $F$ be a number field, $X=\Spec(\mathcal{O}_F)$, $G_F := \Gal(\overline{F}/F)$, $\mathfrak{b} : \Spec(\overline{F}) \rightarrow X$ be a geometric point and $\pi :=\pi_1(X, \mathfrak{b})$. 
Then $\pi \cong \Gal(F_{un}^f/F)$ where $F_{un}^f$ is the maximal extension of $F$ in $\overline{F}$ unramified at all finite primes. 
Denote $\pi^{un} := \Gal(F_{un}/F)$ where $F_{un}$ is the maximal unramified extension of $F$ in $\overline{F}$. 
Let $S_f$ be any finite set of finite primes of $F$, 
$X_{\infty}$ be the set of all real places of $F$, 
$S := S_f \cup X_{\infty}$, 
$\widetilde{S}$ be any set satisfying $S_f \subset \widetilde{S} \subset S$, 
$U_S := \Spec(\mathcal{O}_F\left [ \frac{1}{S_f} \right ])$ and  
$\mathfrak{b}_S : \Spec(\overline{F}) \rightarrow U_S$ be a geometric point and
$\pi_{\widetilde{S}} := \Gal(F_{un}^{\widetilde{S}}/F)$ where $F_{un}^{\widetilde{S}}$ is the maximal extension of $F$ in $\overline{F}$ unramified outside $\widetilde{S}$. 
Then $\pi_1(U_S, \mathfrak{b}_S) \cong \pi_S$. 

\noindent Let $p_S : \pi_S \rightarrow  \pi_{S_f}$ and $\kappa_{\widetilde{S}} : \pi_{\widetilde{S}} \rightarrow \pi^{un}$ be natural quotient maps. For each (possibly infinite) prime $v$ of $F$, let $\pi_v := \Gal(\overline{F_v}/F_v)$, $I_v \subset \pi_v$ be the inertia group and 
$\kappa_v : \pi_v \rightarrow G_F \rightarrow \pi^{un}$ 
be given by choices of embeddings $\overline{F} \rightarrow \overline{F_v}$. 
For each $v \in \widetilde{S}$, let 
$i_v = i_{v, \widetilde{S}} : \pi_v \rightarrow G_F \rightarrow \pi_{\widetilde{S}}$ be given by the embeddings $\overline{F} \rightarrow \overline{F_v}$ same as above and $i_v'=i_{v, S}' : \Spec(F_v) \rightarrow \Spec(F) \rightarrow U_S$ be the natural map.  

\noindent For a $\pi_{\widetilde{S}}$-module $M$, let $M_v$ be $M$ equipped with the $\pi_v$-module structure given by $i_v$. Similarly, for an abelian \'etale sheaf $\mathcal{F}$ on $U_S$, denote the abelian \'etale sheaf on $\Spec(F_v)$ induced by $i_v'$ by $\mathcal{F}_v$. For a finite abelian group $G$, its Pontryagin dual $G^D :=\Hom(G, \Q/\Z)$ is isomorphic to $G$. For a Galois group $G=\Gal(L/K)$ and a $G$-module $M$, denote the dual group of $M$ by $M^{\vee} := \Hom_G(M, L^{\times})$. For a locally constant abelian \'etale sheaf $\mathcal{F}$ on $U_S$, its Cartier dual $\mathcal{F}^{\vee} := \underline{\Hom}_{U_S}(\mathcal{F}, \mathbf{G}_m)$ is also locally constant and $\mathcal{F}^{\vee \vee} \cong \mathcal{F}$.

\subsection{Cohomology with compact support} \label{Sub22}

\noindent We write $r_v$ for the restriction map of cochains or cohomology classes from $\pi_{\tilde S}$ to $\pi_v$ (induced by $i_v$) or from $U_S$ to $\Spec(F_v)$ (induced by $i_v'$). Since the category of abelian \'etale sheaves on $\Spec(F_v)$ is equivalent to the category of $\pi_v$-modules (\cite[Proposition 5.7.8]{LF}), the second map can be also viewed as the map of cochains or cohomology classes from $U_S$ to $\pi_v$. 

\noindent Let $C(G,M)$ be the standard inhomogeneous group cochain complex of a group $G$ with values in $M$. For any finite $\pi_{\widetilde{S}}$-module $M$, we consider the complex defined as a mapping fiber:
$$
C_c(\pi_{\widetilde{S}}, M) :=\operatorname{Fiber}[C(\pi_{\widetilde{S}}, M) \to \prod_{v \in {\widetilde{S}}} C(\pi_v, M_v)]
$$
and 
$$
H_c^r(\pi_{\widetilde{S}}, M):=H^r \left(C_c(\pi_{\widetilde{S}}, M)\right).
$$

\noindent Explicitly, 
$$
C^n_c(\pi_{\widetilde{S}}, M)=C^n(\pi_{\widetilde{S}},M)\times\prod_{v\in {\widetilde{S}}} C^{n-1}(\pi_v, M_v)
$$
and $d(a, (b_v)_{v \in \widetilde{S}})=(da, (r_v(a)-db_v)_{v \in \widetilde{S}})$ for $(a,(b_v)_{v \in {\widetilde{S}}})\in C^n_c(\pi_{\widetilde{S}}, M)$. By the definition of $H_c^r(\pi_{\widetilde{S}}, M)$, there is a long exact sequence
$$\cdots \rightarrow  H^r_c(\pi_{\widetilde{S}}, M) \rightarrow  H^r(\pi_{\widetilde{S}}, M) \rightarrow  \prod_{v \in {\widetilde{S}}} H^r(\pi_v, M_v) \rightarrow  H^{r+1}_c(\pi_{\widetilde{S}}, M) \rightarrow \cdots.
$$

\noindent For an abelian \'etale sheaf $\mathcal{F}$ on $U_S$, the complex 
$$
C_c(U_S, \mathcal{F}) :=\operatorname{Fiber}[C(U_S, \mathcal{F}) \to \prod_{v \in S} C(\pi_v, \mathcal{F}_v)]
$$ 
and
$$ 
H^r_c(U_S, \mathcal{F}) :=  H^r \left(C_c(U_S, \mathcal{F})\right)
$$ 
are defined by the same way and the same long exact sequence holds.

\begin{proposition} \label{prop21}
(\cite[Proposition 2.2.6]{ADT}) $H^2_c(U_S, \mathbf{G}_m)=0$, $H^3_c(U_S, \mathbf{G}_m)=\Q/\Z$ and $H_c^r(U_S, \mathbf{G}_m)=0$ for all $r>3$. 
\end{proposition}

\begin{remark} \label{rmk22}
(\cite[Remark 2.2.8(a)]{ADT}) Let $\mathcal{O}_F^{\times, +}$ be the group of totally positive units of $F$, and $Cl^+(F)$ be the narrow class group of $F$. Then $H^0_c(X, \mathbf{G}_m)=\mathcal{O}_F^{\times, +}$ and $H^1_c(X, \mathbf{G}_m)=Cl^+(F)$.
\end{remark}

\begin{theorem} \label{thm23}
(\cite[Theorem 2.3.1]{ADT}; Artin-Verdier duality) Let $\mathcal{F}$ be a constructible sheaf on an open subscheme $U$ of $X$. The Yoneda pairing 
\begin{equation*}
H^r_c(U,\mathcal{F}) \times Ext_U^{3-r}(\mathcal{F}, \mathbf{G}_{m}) \rightarrow  H^3_c(U, \mathbf{G}_{m}) \cong \Q/\Z
\end{equation*}
is a nondegenerate pairing of finite abelian groups for all $r \in \Z$. 
\end{theorem}

\noindent For an abelian group $A$, let $A_{U_S}$ be the constant sheaf on the \'etale site $\mathbf{Et}(U_S)$ (defined in Section \ref{Sub22}) defined by $A$.

\begin{proposition} \label{prop24}
$H^3_c(U_S, \Z/n\Z) \cong \mu_n(F)^D$. 
\end{proposition}

\begin{proof}
By Artin-Verdier duality, 
$$
H^3_c(U_S, \Z/n\Z) \cong \Hom_{U_S}(\Z/n\Z, \mathbf{G}_{m})^D.
$$
The isomorphism $(\Z/n\Z)_{U_S}^{\vee} \cong \ker(\Z_{U_S}^{\vee} \overset{n}{\rightarrow} \Z_{U_S}^{\vee}) = \ker(\mathbf{G}_{m, U_S} \overset{n}{\rightarrow} \mathbf{G}_{m, U_S}) = \mu_{n, U_S}$ implies
\begin{equation*}
\Hom_{U_S}(\Z/n\Z, \mathbf{G}_{m}) = (\Z/n\Z)_{U_S}^{\vee}(U) = \mu_{n, U_S}(U_S) = \mu_n(\mathcal{O}_F\left [ \frac{1}{S_f} \right ]) = \mu_n(F). 
\qedhere
\end{equation*}
\end{proof}

\noindent Denote the isomorphism $H^3_c(U_S, \Z/n\Z) \rightarrow \mu_n(F)^D$ by $\inv_S$ and $H^3_c(X, \Z/n\Z) \rightarrow \mu_n(F)^D$ by $\inv$. Note that the following diagram commutes and the map $H^3_c(X, \Z/n\Z) \rightarrow H^3_c(U_S,\Z/n\Z)$ is an isomorphism. 

\[
\begin{tikzcd}
H^3_c(X, \Z/n\Z)  \arrow[dr, "\inv", "\simeq"'] \arrow[rr] && H^3_c(U_S,\Z/n\Z) \arrow[dl, "\inv_S"', "\simeq"] \\
 & \mu_n(F)^D & 
\end{tikzcd}
\]

\subsection{Arithmetic Chern-Simons action without boundary} \label{Sub23}

\noindent Let $\mathbf{FSet}_{\pi_S}$ be the category of finite continuous $\pi_S$-sets, $\mathbf{FEt}(U_S)$ be the category of finite \'etale $U_S$-schemes and $\mathbf{Et}(U_S)$ be the category of \'etale $U_S$-schemes. 
Each category can be understood as a site with a natural Grothendieck topology. 
Since the functor $\mathbf{FEt}(U_S) \rightarrow \mathbf{FSet}_{\pi_S}$ ($Y \mapsto Y(\mathfrak{b}_S) := \Hom_{U_S}(\Spec(\overline{F}), Y)$) is an equivalence of categories (\cite[Theorem 3.2.12]{LF}) and there is a natural morphism of sites $\mathbf{FEt}(U_S) \rightarrow \mathbf{Et}(U_S)$, the map
$$ 
j^i : H^i(\pi_S, \Z/n\Z) \cong H^i(\mathbf{FEt}(U_S), \Z/n\Z) \rightarrow H^i(\mathbf{Et}(U_S), \Z/n\Z) = H^i(U_S, \Z/n\Z)
$$
is defined and it induces a homomorphism
$$
j^i_c : H^i_c(\pi_S, \Z/n\Z) \rightarrow H^i_c(U_S, \Z/n\Z).
$$
It is easy to show that the map $H^3(\pi^{un}, \Z/n\Z) \rightarrow H^3_c(\pi, \Z/n\Z)$ given by
$$
[w] \mapsto [(\kappa_{X_{\infty}}^*(w), (0)_{v \in X_{\infty}})]
$$
is a well-defined group homomorphism. 
(Note that $\kappa_{X_{\infty}}:\pi \cong \Gal(F_{un}^f/F) \to \pi^{un}=\Gal(F_{un}/F)$ is the natural quotient map, i.e. $\kappa_{X_{\infty}}$ is the map $\kappa_{\widetilde{S}}$ defined in Section \ref{Sub21} for $S_f =\phi$ and $\widetilde{S} = X_{\infty}$.)
Now define the map $j^3_{un}$ by
\begin{eqnarray}\label{jun}
j^3_{un} : H^3(\pi^{un}, \Z/n\Z) \rightarrow H^3_c(\pi, \Z/n\Z)  \xrightarrow{j^3_c} H^3_c(X, \Z/n\Z).
\end{eqnarray}

\noindent Let $A$ be a finite group, $\mathcal{M}(A):=\Hom_{cont}(\pi^{un}, A)/A$ and fix a class $c \in H^3(A, \Z/n\Z)$ for an integer $n \geq 2$. Define the \textbf{arithmetic Chern-Simons action} by a function
$$ CS_c : \mathcal{M}(A) \rightarrow \mu_n(F)^D, \quad \,\,\, [\rho] \mapsto \inv(j^3_{un}(\rho^*(c))). $$
Since $\rho^*(c)$ depends only on the class $[\rho]$, $CS_c$ is well-defined.

\begin{proposition} \label{prop25}
Let $m>0$ be a divisor of $n$ such that $\mu_n(F)=\mu_m(F)$ (for example, $m=\left | \mu_n(F) \right |$), $\alpha : \Z/n\Z \rightarrow \Z/m\Z$ be the group homomorphism defined by $\alpha(1)=1$ and $c'=\alpha_* c \in H^3(A, \Z/m\Z)$. Then $CS_c([\rho])=CS_{c'}([\rho])$ for every $[\rho] \in \mathcal{M}(A)$. 
\end{proposition}

\begin{proof}
The proposition follows from the following commutative diagram. 
\[
\begin{tikzcd}[column sep=1.25em]
H^3(A, \Z/n\Z) \arrow[r, "\rho^*"] \arrow[d, "\alpha_*"]
& H^3(\pi^{un}, \Z/n\Z) \arrow[r, "j^3_{un}"] \arrow[d, "\alpha_*"]
& H^3_c(X, \Z/n\Z) \arrow[r, "\simeq"] \arrow[d, "\alpha_*"]
& \Hom_X(\Z/n\Z, \mathbf{G}_{m, X})^D  \arrow[r, "\simeq"] \arrow[d, "(\alpha^*)^D"]
& \mu_n(F)^D \arrow[d, equal]  \\
H^3(A, \Z/m\Z) \arrow[r, "\rho^*"]
& H^3(\pi^{un}, \Z/m\Z) \arrow[r, "j^3_{un}"]
& H^3_c(X, \Z/m\Z) \arrow[r, "\simeq"]
& \Hom_X(\Z/m\Z, \mathbf{G}_{m, X})^D \arrow[r, "\simeq"]
& \mu_m(F)^D
\end{tikzcd}
\] 
\end{proof}

\noindent By the proposition above, we may assume that $\mu_n(\overline{F}) =\mu_n(F)$ which we assume from now on. Since $F$ is totally imaginary for $n \geq 3$ under this assumption, the case $n=2$ is the only case which \cite{ACST} does not cover. 
Note that there is an isomorphism between $\mu_n(F)^D$ with $\frac{1}{n}\Z/\Z$ so that we can regard $CS_c$ a $\frac{1}{n}\Z/\Z$-valued function.
Later when we state and prove the decomposition formula, we will choose an identification $\phi: \mu_n(F)^D\xrightarrow{\sim} \frac{1}{n}\Z/\Z$ (see (2), Section \ref{Sec3}).

\subsection{Arithmetic Chern-Simons action with boundary} \label{Sub24}

\noindent Let $S_n$ be any finite set of finite primes of $F$ containing primes of $F$ above $n$. Denote $T=S_n \cup X_{\infty}$ and $T_f=S_n$. Note that $T=T_f=S_n$ for $n \geq 3$ due to the assumption $\mu_n(\overline{F}) =\mu_n(F)$. 
For $T' \in \left \{ T, T_f \right \}$, let 
$$
Y_{T'}(A) := \Hom_{cont}(\pi_{T'}, A)
$$
equipped with the conjugation action of $A$ and denote its action groupoid by $\mathcal{M}_{T'}(A)$. 
Similarly, let
$$
Y^{loc}_{T'}(A) := \prod_{v \in T'} \Hom_{cont}(\pi_v, A)
$$
equipped with the conjugation action of $A^{T'} := \prod_{v \in T'} A$ and denote its action groupoid by $\mathcal{M}^{loc}_{T'}(A)$. 
Define the functor $r_{T'} : \mathcal{M}_{T'}(A) \rightarrow \mathcal{M}^{loc}_{T'}(A)$ by 
$$
r_{T'}(\rho) = i^*_{T'}(\rho) := (\rho \circ i_v)_{v \in T'}. 
$$
Let $C_T^i := \prod_{v \in T} C^i(\pi_v, \Z/n \Z)$, 
$Z_T^i := \prod_{v \in T} Z^i(\pi_v, \Z/n \Z)$ and 
$B_T^i := \prod_{v \in T} B^i(\pi_v, \Z/n \Z)$ be the product of continuous cochains, cocycles and coboundaries for each $v \in T$ and $d_T := (d_v)_{v \in T} : C_T^2 \rightarrow Z_T^3$. Let $\rho_T := (\rho_v)_{v \in T} \in Y_T^{loc}(A)$, $c \circ \rho_T := (c \circ \rho_v)_{v \in T}$ and $c \circ Ad_a := (c \circ Ad_{a_v})_{v \in T}$ for $a=(a_v)_{v \in T} \in A^T$ where $Ad_{a_v} : A \rightarrow A$ is defined by $Ad_{a_v}(x)=a_v x a_v^{-1}$. For $H_T^2 := \prod_{v \in T}H^2(\pi_v, \Z/n \Z)$, 
$$ H(\rho_T) := d_T^{-1}(c \circ \rho_T)/B_T^2 \subset C_T^2/B_T^2 $$
is a $H_T^2$-torsor. 

\noindent By \cite[Theorem 2.5.2]{GC}, there is an isomorphism $\inv_v : H^2(\pi_v, \Z/n \Z) \rightarrow \frac{1}{n}\Z/\Z$ for all $v \in T$. By the Poitou-Tate exact sequence (\cite[Theorem I.4.10]{ADT}), the sequence
\begin{eqnarray}\label{exactseq}
H^2(\pi_T, \Z/n \Z) \rightarrow H_T^2 \rightarrow H^0(\pi_T, \Z/n \Z^{\vee})^D \rightarrow 0
\end{eqnarray}
is exact. Denote the map $H_T^2 \rightarrow H^0(\pi_T, \Z/n \Z^{\vee})^D \xrightarrow{\simeq} \frac{1}{n}\Z/\Z$ by $\sum$ and let
\begin{eqnarray}\label{LT}
L_T(\rho_T) := \textstyle \sum_*(H(\rho_T))
\end{eqnarray}
be the $\frac{1}{n}\Z/\Z$-torsor defined by the pushout of the $H_T^2$-torsor $H(\rho_T)$ by $\sum$. 
By \cite[Proposition 8.3.18]{NSW} and \cite[Theorem 10.6.1]{NSW}, $H^3(\pi_{T_f}, \Z/n \Z)=0$ so for $\rho \in Y_{T_f}(A)$, $c \circ \rho = d \beta$ for some $\beta \in C^2(\pi_{T_f}, \Z/n \Z)$. Then for the quotient map $p_T : \pi_T \rightarrow \pi_{T_f}$, we have $c \circ \rho \circ p_T = d(\beta \circ p_T)$ and $d(i_T^*(\beta \circ p_T)) = c \circ i_T^*(\rho \circ p_T)$ so $[i_T^*(\beta \circ p_T)] \in H(i_T^*(\rho \circ p_T))$. Now define the \textbf{arithmetic Chern-Simons action} with boundary $T_f$ by
\begin{eqnarray}\label{wb}
CS_{T_f, c}(\rho):=\textstyle{\sum_*} ([i_T^*(\beta \circ p_T)]) \in L_T(i_T^*(\rho \circ p_T)), \quad \rho \in Y_{T_f}(A).
\end{eqnarray}
For $\beta, \beta' \in C^2(\pi_{T_f}, \Z/n \Z)$ such that $d \beta = d \beta' = c \circ \rho$, $\beta'=\beta+z$ for some $z \in Z^2(\pi_{T_f}, \Z/n \Z)$. Then $\textstyle{\sum_*}([i_T^*(z \circ p_T)])=0$ by the Poitou-Tate exact sequence \eqref{exactseq}, so $CS_{T_f, c}$ is well-defined. 
Because $H^3(\pi_T, \Z/n\Z)\neq 0$ when $X_\infty$ is not empty (i.e. there is a real place in $F$), the arithmetic CS action with boundary does not seem to extend from $Y_{T_f}(A)$ to $Y_{T}(A)$.

\begin{remark} \label{rmk26}
The map $H_T^2 \rightarrow H^0(\pi_T, \Z/n \Z^{\vee})^D$ is the dual of $H^0(\pi_T, \Z/n \Z^{\vee}) \rightarrow \prod_{v \in T} H^0(\pi_v, \Z/n \Z^{\vee})$ induced by $i_{v, T}$ for each $v \in T$. (Each map corresponds to the map $\gamma^2$ and $\beta^0$ in
\cite[Theorem I.4.10]{ADT}, respectively. In \cite{ADT}, the map $\gamma^r$ is defined to be the dual of $\beta^{2-r}$ for each $r \in \Z$.) 
For each $v \in T$, the map 
$H^0(\pi_T, \Z/n \Z^{\vee}) \rightarrow H^0(\pi_v, \Z/n \Z^{\vee})$ is identified with an isomorphism $\mu_n(F) \rightarrow \mu_n(F_v)$ so the map $\sum$ is same as $\sum_{v \in T} \inv_v$ and equals to the sum map $\sum$ of \cite[p. 3]{ACST}. 
\end{remark}

Note that $L_T$ in \eqref{LT} can be extended to a functor from $\mathcal{M}_T^{loc}(A)$ to the category of $\frac{1}{n}\Z/\Z$-torsors $\mathbf{Tors}(\frac{1}{n}\Z/\Z)$ exactly as in \cite[p. 6-7]{ACST}. 
Denote 
$$L_T^{glob} := L_T \circ r_T : \mathcal{M}_T(A) \rightarrow \mathbf{Tors}(\frac{1}{n}\Z/\Z). $$
According to the axioms of topological field theory, it is natural to expect that $CS_{T_f,c}(\cdot)$ is an invariant section of the functor $L_T^{glob} \circ p_T^*$ (where $p_T^*:\mathcal{M}_{T_f}(A) \to \mathcal{M}_{T}(A)$ is induced from $p_T$); we provide such a result.
\begin{lemma} \label{lem27}
Let $\rho \in Y_T(A)$ and $a \in \Aut(\rho)$ be a morphism in $\mathcal{M}_T(A)$. Then $L_T^{glob}(a)=id_{L_T^{glob}(\rho)}$. 
\end{lemma}

\begin{proof}
By the definition of the action groupoid $\mathcal{M}_T(A)$, $Ad_a \circ \rho = \rho$. By \cite[Lemma A.2]{ACST}, there is $h_a \in C^2(A, \Z/n \Z)/B^2(A, \Z/n \Z)$ such that
$$
c \circ Ad_a = c + dh_a, 
$$
so $dh_a \circ \rho = d(h_a \circ \rho)=0$ and $h_a \circ \rho \in H^2(\pi_T, \Z/n \Z)$. $L_T^{glob}(a)$ maps an element of $L_T^{glob}(\rho)$ represented by $(f,x) \in H(r_T(\rho)) \times \frac{1}{n}\Z/\Z$ to an element represented by $(f+(h_a \circ \rho \circ i_v)_{v \in T}, x)$, or equivalently, $(f, x+\sum((h_a \circ \rho \circ i_v)_{v \in T}))$. By the exact sequence \eqref{exactseq} above, $\sum ((h_a \circ \rho \circ i_v)_{v \in T})=0$ so $L_T^{glob}(a)$ is the identity map on $L_T^{glob}(\rho)$. 
\end{proof}

\noindent By the above lemma, $L_T^{glob}$ restricted on the orbit $[\rho] \in \mathcal{M}_T(A)=Y_T(A)/A$ is a functor from a connected groupoid to $\mathbf{Tors}(\frac{1}{n}\Z/\Z)$ with no holonomy. Therefore, the set $L_T^{inv}([\rho])$ of its invariant sections becomes a non-zero $\frac{1}{n}\Z/\Z$-torsor and it is given explicitly by
$$
L_T^{inv}([\rho]):=\mathop{\lim_{\longleftarrow}}_{\rho' \in [\rho]} L_T^{glob}(\rho')= \big\{ (x_{\rho'}) \in  \prod_{\rho' \in [\rho]} L_T^{glob}(\rho') \mid \forall a \in \Hom(\rho_1, \rho_2) \,\, L_T^{glob}(a) (x_{\rho_1})=x_{\rho_2} \big\}. 
$$
Thus, for a given $[\rho] \in \mathcal{M}_{T_f}(A)=Y_{T_f}(A)/A$, we have   $[p_T^*(\rho)]=[\rho \circ p_T] \in \mathcal{M}_T(A)$ and $L_T^{inv}([\rho \circ p_T])$ is a $\frac{1}{n}\Z/\Z$-torsor;
one can check the desired result
$$
CS_{T_f, c}([\rho])\in L_T^{inv}([\rho \circ p_T]), \quad [\rho] \in \mathcal{M}_{T_f}(A).
$$

\section{Decomposition formula} \label{Sec3}

\noindent An explicit computation of the arithmetic Chern-Simons invariant relies on the decomposition formula, which expresses the arithmetic Chern-Simons invariant as the difference of a ramified global trivialization and an unramified local trivialization. In this section, we prove the decomposition formula. The proof is done in several steps. \\

\noindent (1) By \cite[8.6.10 (ii)]{NSW}, $H^3(\pi_T, \Z/n \Z) \cong \prod_{v \in T} H^3(\pi_v, \Z/n \Z)$ so the top row of the following diagram is exact. Its bottom row is exact by Poitou-Tate exact sequence and Remark \ref{rmk26}. 

\[
\begin{tikzcd}
{H^2(\pi_T, \Z/n \Z)} \arrow[r] \arrow[d, equal] & {\prod_{v \in T} H^2(\pi_v, \Z/n \Z)} \arrow[r] \arrow[d, equal] & {H^3_c(\pi_T, \Z/n \Z)} \arrow[r] \arrow[d, dashed, "\inv'_T"] & 0 \\
{H^2(\pi_T, \Z/n \Z)} \arrow[r] & {\prod_{v \in T} H^2(\pi_v, \Z/n \Z)} \arrow[r, "\sum_{v \in T} \inv_v"] & \frac{1}{n}\Z/\Z \arrow[r] & 0
\end{tikzcd}
\]

\noindent From the diagram above, we obtain an isomorphism $\inv'_T : H^3_c(\pi_T, \Z/n \Z) \rightarrow \frac{1}{n}\Z/\Z$ which makes the diagram commute. \\

\noindent (2) By \cite[Theorem 6.1.1]{CM}, $j^3_c : H^3_c(\pi_T, \Z/n \Z) \rightarrow H^3_c(U_T, \Z/n \Z)$ is an isomorphism. Now let $\phi : \mu_n(F)^D \rightarrow \frac{1}{n}\Z/\Z$ be the isomorphism defined by the composition
$$
\mu_n(F)^D \xrightarrow{\inv_T^{-1}} 
H^3_c(U_T, \Z/n \Z) \xrightarrow{(j^3_c)^{-1}}
H^3_c(\pi_T, \Z/n \Z) \xrightarrow{\inv'_T}
\frac{1}{n}\Z/\Z.
$$
Then the following diagram commutes:

\[
\begin{tikzcd}
 & & {H^3_c(X, \Z/n \Z)} \arrow[d, "\simeq"'] \arrow[rd, "\inv", "\simeq"'] & \\
{H^3(\pi^{un}, \Z/n \Z)} \arrow[d] \arrow[r] \arrow[rru, "j^3_{un}"] & {H^3_c(\pi, \Z/n \Z)} \arrow[d] \arrow[ru, "j^3_c"'] & {H^3_c(U_T, \Z/n \Z)} \arrow[r, "\inv_T", "\simeq"'] & \mu_n(F)^D \arrow[d, dashed, "\phi", "\simeq"'] \\
{H^3_c(\pi_{T_f}, \Z/n \Z)} \arrow[r] & {H^3_c(\pi_T, \Z/n \Z)} \arrow[ru, "\simeq"', "j^3_c"] \arrow[rr, "\inv'_T", "\simeq"'] &  & \frac{1}{n}\Z/\Z.
\end{tikzcd}
\]

\noindent (3) By \cite[Proposition 2.18]{GC}, $H^2(\pi_v/I_v, \Z/n \Z)=H^3(\pi_v/I_v, \Z/n \Z)=0$ for every $v \in T$. Let $T_1 \subset T_2$ be finite sets of primes of $F$ which may or may not contain real places and denote the projection $\pi_{T_2} \rightarrow \pi_{T_1}$ by $\kappa_{T_1, T_2}$. The map $H^3_c(\pi_{T_1}, \Z/n \Z) \rightarrow H^3_c(\pi_{T_2}, \Z/n \Z)$ can be described as follows. 

\noindent Choose an element $[(a, (b_v)_{v \in T_1})] \in H^3_c(\pi_{T_1}, \Z/n \Z)$. For $v \in T_2 \setminus T_1$, $\kappa_v$ factors through $\pi_v/I_v$ so $r_v(a)$ can be restricted to $r_v(a) \mid_{\pi_v/I_v} \in Z^3(\pi_v/I_v, \Z/n \Z)=B^3(\pi_v/I_v, \Z/n \Z)$. Now there is a canonical
$$
\widetilde{b_v} \in C^2(\pi_v/I_v, \Z/n \Z)/B^2(\pi_v/I_v, \Z/n \Z) \,\,\, (v \in T_2 \setminus T_1)
$$
such that $d\widetilde{b_v} = r_v(a) \mid_{\pi_v/I_v}$. This can be lifted to a canonical class $b_v \in C^2(\pi_v, \Z/n \Z)/B^2(\pi_v, \Z/n \Z)$. (Note that $b_v=0$ if $v \in X_{\infty}$.) Now $[(\kappa_{T_1, T_2}^*(a), (b_v)_{v \in T_2})]$ is an element of $H^3_c(\pi_{T_2}, \Z/n \Z)$ and it is independent of the choice of $b_v$ because $H^2(\pi_v/I_v, \Z/n \Z)=0$. The canonical map
$$
H^3_c(\pi_{T_1}, \Z/n \Z) \rightarrow H^3_c(\pi_{T_2}, \Z/n \Z)  \,\,\, ([(a, (b_v)_{v \in T_1})] \mapsto [(\kappa_{T_1, T_2}^*(a), (b_v)_{v \in T_2})])
$$
is a group homomorphism and one can deduce that it is same as the map induced by $C^3(\pi_{T_1}, \Z/n \Z) \rightarrow C^3(\pi_{T_2}, \Z/n \Z)$ from the definition of the mapping fiber. \\

\noindent (4) Let $w$ be the cocycle representing $\rho^*(c) \in H^3(\pi^{un}, \Z/n \Z)$. By the vanishing of $H^3(\pi_{T_f}, \Z/n \Z)$, we can find a global cochain $b'_+ \in C^2(\pi_{T_f}, \Z/n \Z)$ such that $\kappa_{T_f}^*(w)=db'_+$ (here $\kappa_{T_f}:\pi_{T_f} \to \pi^{un}$ is the natural quotient).  Denote $b_+ := b'_+ \circ p_T \in C^2(\pi_T, \Z/n \Z)$ and $b_{+,v} := r_v(b_v) \in C^2(\pi_v, \Z/n \Z)$. 

\noindent From the discussion in (3), the image of $\rho^*(c)=[w]$ in $H^3_c(\pi_T, \Z/n \Z)$ is $[(\kappa_T^*(w), (b_{-,v})_{v \in T})]$, where $b_{-,v}$ is a lift of $\widetilde{b_{-,v}} \in C^2(\pi_v/I_v, \Z/n \Z)$ such that $\kappa_v^*(w) \mid_{\pi_v/I_v} = d\widetilde{b_{-,v}}$. Now
$$
[(\kappa_T^*(w), (b_{-,v})_{v \in T})] =[(db_+, (b_{-,v})_{v \in T})] =[(0, (b_{-,v}-b_{+,v})_{v \in T})]
$$
and by the diagram in (2) and the fact that $b_{-,v}=b_{+,v}=0$ for $v \in X_{\infty}$, 
$$
CS_c([\rho]) = \inv'_T([(0, (b_{-,v}-b_{+,v})_{v \in T})]) = \sum_{v \in T_f} \inv_v([b_{-,v}-b_{+,v}]). 
$$

\begin{theorem} \label{thm3}
Let $F$ be a number field and $n \geq 2$ be an integer. For any finite set $T_f$ of finite primes of $F$ containing all primes dividing $n$ with the notations above, we have the decomposition formula
$$
CS_c([\rho]) := \inv(j^3_{un}(\rho^*(c))) = \sum_{v \in T_f} \inv_v([b_{-,v}-b_{+,v}]).
$$
\end{theorem}

\begin{proof}
For $\beta_v := b_{-,v}$, $\sum_{v \in T} (\beta_v) := \textstyle{\sum_*}((\beta_v)_{v \in T}) \in L_T^{inv}([\rho \circ \kappa_T])$ 
is well-defined and $CS_{T_f, c}([\rho \circ \kappa_{T_f}])$ is also an element of a $\frac{1}{n}\Z/\Z$-torsor $L_T^{inv}([\rho \circ \kappa_T])$ so their difference is an element of $\frac{1}{n}\Z/\Z$. 
It is easy to show that $CS_{T_f, c}([\rho \circ \kappa_{T_f}])=\textstyle{\sum_*}((b_{+,v})_{v \in T})$ and 
\begin{equation*}
CS_c([\rho]) = \sum_{v \in T_f} \inv_v([b_{-,v}-b_{+,v}]) = \sum_{v \in T} (\beta_v) - CS_{T_f, c}([\rho \circ \kappa_{T_f}]).
\qedhere
\end{equation*}
\end{proof}

\section{Examples} \label{Sec4}
In this section, we compute the arithmetic Chern-Simons invariants in certain cases. 
In the first 2 subsections, we concentrate on the case $n=2$ and provide some explicit examples of the computation of $CS_c([\rho])$. 
In the last subsection, we prove that for $n \geq 2$ and a finite group $G$, there are infinitely many number fields $F$ whose arithmetic Chern-Simons invariants associated to the gauge group $A=G \rtimes \Z/n \Z$ (the semidirect product of $G$ by $\Z/n \Z$) are non-vanishing, which generalizes the result of \cite{BCG}. 

\subsection{The totally real analogue with $n=2$}\label{Sub41}

Since most of the results in this section are analogues to Section 5 of \cite{ACST}, proofs are generally omitted there. 
Under the assumption $\mu_n(\overline{F}) =\mu_n(F)$, (5.1)-(5.5) of \cite{ACST} can be extended to every number field $F$. 
Since we use $\pi^{un}$ rather than $\pi$, Assumption 5.7 and 5.9 of \cite{ACST} should be refined by adding conditions about unramifiedness at real places.

\begin{assumption} \label{ass41}
(1) $c = \alpha \cup \epsilon \in H^3(A, \Z/2\Z)$ with surjective $\alpha : A \rightarrow \Z/2\Z$, and $\epsilon \in H^2(A, \Z/2\Z)$ corresponding to a central extension 
$$ E : 0 \rightarrow \Z/2\Z \rightarrow \Gamma \overset{\varphi}{\rightarrow} A \rightarrow 1. $$
(2) There are Galois extensions of $F$:
$$ F \subset F^{\alpha} \subset F^- \subset F^+ $$
such that $\Gal(F^-/F) \cong A$, $\Gal(F^+/F) \cong \Gamma$, $F^-/F$ is unramified at all primes (including real primes) and $F^+/F$ is unramified at the primes above $2$. \\
(3) $F^{\alpha}$ is the fixed field of the kernel of the composition $\Gal(F^-/F) \xrightarrow{\simeq} A \overset{\alpha}{\rightarrow} \Z/2\Z$. \\
(4) $\rho : \pi^{un} \rightarrow A$ is given by the composition $\pi^{un} \rightarrow \Gal(F^-/F) \xrightarrow{\simeq} A$, where $\pi^{un} \rightarrow \Gal(F^-/F)$ is the natural projection map. 
\end{assumption}

\begin{remark} \label{rmk42}
Let $B=\left \{ F_1, \cdots, F_m \right \}$ be the set of subfields of $F^-$ which are quadratic extensions of $F$. Then $B$ is in bijection with the set of surjective homomorphisms $\Gal(F^-/F) \rightarrow \Z/2\Z$. So for each $F_i \in B$, there is a unique surjective homomorphism $\alpha_i : \Gal(F^-/F) \rightarrow \Z/2\Z$ such that $F^{\alpha_i}=F_i$. 
\end{remark}

\noindent Let $S_f$ be the set of finite primes of $F$ ramified in $F^+$, $S_2$ be the set of primes of $F$ above $2$ ($S_f \cap S_2 = \phi$ by the assumption) and $T_f := S_f \cup S_2$. Denote the map $\gamma$ of \cite[Lemma 5.2]{ACST} for the settings (a) and (b) below by $\gamma_{+}$ and $\gamma_{-,v}$ (for each $v \in T_f$), respectively. \\
(a) $\widetilde{A}=\pi_T$, $f=\rho \circ \kappa_T$, 
$\widetilde{f} : \pi_T \rightarrow \Gal(F^+/F) \cong \Gamma$ where $\pi_T \rightarrow \Gal(F^+/F)$ is the projection. \\
(b) $\widetilde{A}=\pi_v$, $f= \rho \circ \kappa_v$, $\widetilde{f} : \pi_v \rightarrow \pi_v/I_v \overset{f'}{\rightarrow} \Gamma$ for $f'$ satisfying $\varphi \circ f' = (\rho \circ \kappa_v) \mid_{\pi_v/I_v} \in \Hom(\pi_v/I_v, A)$. \\

\noindent Denote $\gamma_{+,v} := \gamma_{+} \circ i_{v, T} \in \Hom(\pi_v, \Z/2\Z)$ and $\psi_v := \gamma_{+,v}-\gamma_{-,v}$ for each $v \in T_f$. Note that $\gamma_{-,v}$ is unramified since it factors through $\pi_v/I_v$. For $v \in X_{\infty}$, $\kappa_v=0$ so $\inv_v(\rho_v^*(\alpha) \cup \psi_v)=0$ for $\rho_v := \rho \circ \kappa_v$.

\begin{theorem} \label{thm43}
Suppose Assumption \ref{ass41} holds. Then,
\begin{equation*}
CS_c([\rho])=\sum_{v \in T_f} \inv_v(\rho_v^*(\alpha) \cup \psi_v) \equiv \frac{r}{2} \,\, (mod \,\, \Z)
\end{equation*}
where $r$ is the number of primes in $S_f$ which are inert in $F^{\alpha}$. 
\end{theorem}

\begin{assumption} \label{ass44}
$K \subset L$ and $E$ are number fields, $D>1$ is a squarefree integer such that \\
(1) $\Gal(L/\Q) \cong \Gamma$ and $d_L$ is odd. \\
(2) $\Gal(K/\Q) \cong A$ and $K$ is totally real. \\
(3) $[E:\Q]$ is odd and $E \cap L  = \Q$.  \ \\
(4) $D$ divides $d_K$, $\Q(\sqrt{D}) \subset K$ and $K/\Q(\sqrt{D})$ is unramified at all primes.
\end{assumption}

\begin{proposition} \label{prop45}
Let $t>1$ be a squarefree integer prime to $D$ and $F=\Q(\sqrt{Dt})E$ satisfies 
$\Q(\sqrt{Dt}) \cap L =\Q$. 
Then $F^-=KF$ and $F^+=LF$ satisfy Assumption \ref{ass41}. 
\end{proposition}

\begin{proof}
Denote $E' = \Q(\sqrt{Dt})$. By the relations $E \cap L = E' \cap L = E \cap E' = \Q$ and the fact that $E'/\Q$ and $L/\Q$ are Galois, 
$$
[EE' \cap L : \Q]=\frac{[EE': \Q][L : \Q]}{[EE'L : \Q]} = \frac{[E: \Q][E': \Q][L : \Q]}{[EE'L : \Q]}=\frac{[EL: \Q][E'L: L]}{[(EL)(E'L): \Q]}=1
$$
(the last equality is due to $EL \cap E'L=L$ from $([EL:L], [E'L:L])=([E:\Q], 2)=1$) so $F \cap L=\Q$. 
For each real prime $v$ of $F$, $K/\Q$ is unramified at $v$ so $F^-/F$ is also unramified at $v$. Other conditions can be proved exactly as \cite[Proposition 5.10]{ACST}.
\end{proof}

\noindent Now suppose that $F^{\alpha}=F(\sqrt{M}) \subset F^-$ for some divisor $M>0$ of $D$. Denote $\Q_1 = \Q(\sqrt{M})$, $\Q_2 = \Q(\sqrt{N})$ for $N=Dt/M$ and $D_{L/K} := Nm_{K/\Q}(\Delta_{L/K})=d_L/d_K^2$.

\begin{theorem} \label{thm46}
Suppose Assumption \ref{ass44} holds and choose $F \subset F^{\alpha} \subset F^- \subset F^+$ as above, $\rho$ and $c$ as Assumption \ref{ass41}. Then,
\begin{equation*}
CS_c([\rho]) \equiv \frac{s}{2} \,\, (mod \,\, \Z)
\end{equation*}
where $s$ is the number of prime divisors of $(D_{L/K}, D)$ which are inert either in $\Q_1$ or in $\Q_2$. 
\end{theorem}

\begin{proof}
We use Theorem \ref{thm43}. Following the proof of \cite[Theorem 5.13]{ACST}, we can deduce that
$$ S_f =\left \{  \mathfrak{p} \in \Spec(\mathcal{O}_F) : \mathfrak{p} \mid D_{L/K}, \, \mathfrak{p} \nmid t \right \}. $$
Let $F_0=\Q(\sqrt{Dt})$, $\mathfrak{p} \in S_f$ be a prime above $p$ and $\mathfrak{p}_0=\mathfrak{p} \cap F_0$ be a prime of $F_0$. 
Suppose $p$ does not divide $D$. 
If $p$ splits in $F_0$, then $p\mathcal{O}_{F}=\mathfrak{p}_1 \cdots \mathfrak{p}_m$ for some even integer $m$ and $\mathfrak{p}_i$ are all inert or all split in $F^{\alpha}$, so they do not change the parity of $r$ of Theorem \ref{thm43}. 
If $p$ is inert in $F_0$, then $\mathfrak{p}_0$ splits in $F_0(\sqrt{M})$ by \cite[Lemma 5.12]{ACST} and $[F : F_0]=[E:\Q]$ is odd so $\mathfrak{p}$ also splits in $F^{\alpha}$. (Note that $E \cap F_0=\Q$.) \\
Now suppose that $p$ divides $D$. Then $\mathfrak{p}$ is the only prime of $F$ above $p$. Since $[F : F_0]$ is odd, $\mathfrak{p}$ is inert in $F^{\alpha}$ if and only if $\mathfrak{p}_0$ is inert in $F_0(\sqrt{M})$. This is equivalent to the fact that $p$ is inert either in $\Q_1$ or in $\Q_2$.
\end{proof}

\begin{remark} \label{rmk47}
In Proposition 5.10 of \cite{ACST}, $F=\Q(\sqrt{-\left | D \right | \cdot t})$ can be generalized to $F=\Q(\sqrt{-\left | D \right | \cdot t}) E$ for every number field $E$ of odd degree such that $E \cap L = \Q$. 
\end{remark}

\vspace{1em}

Now we give explicit computations of $CS_c$ for $A = \Z/2\Z$, $A = \Z/2\Z \times \Z/2\Z$ and $A=S_4$. 

\textbf{Case I. $A=\Z/2\Z$}: Let $A=\Z/2\Z$, $\Gamma = \Z/4\Z$ and $p$ be a prime congruent to $1$ modulo $4$. Then $K=\Q(\sqrt{p})$, the degree 4 subfield $L$ of $\Q(\mu_p)$ and $D=p$ satisfy Assumption \ref{ass44}. Let $F=\Q(\sqrt{pt})E$ for a squarefree integer $t>1$ prime to $p$ and a number field $E$ of odd degree.  

\begin{proposition} \label{prop48}
Let $F^{\alpha}=F^-=FK$, $F^+=LF$ and $\rho$ and $c$ be chosen as above. Then,
\begin{equation*}
CS_c([\rho]) = \frac{1}{2} \Longleftrightarrow  \left ( \frac{t}{p} \right )=-1.
\end{equation*}
\end{proposition}

\textbf{Case II. $A=\Z/2\Z \times \Z/2\Z$}: Let $A=\Z/2\Z \times \Z/2\Z$, $\Gamma=\mathcal{Q}_8$ (the quaternion group) and $d_1, d_2>1$ be squarefree integers such that $d_1 \equiv d_2 \equiv 1 \,\, (mod \,\, 4)$ and $(d_1, d_2)=1$. 
Suppose that there is a number field $L$ containing $\Q(\sqrt{d_1}, \sqrt{d_2})$ such that $\Gal(L/\Q) \cong \mathcal{Q}_8$ and $d_L$ is odd. 
Then $K=\Q(\sqrt{d_1}, \sqrt{d_2})$, $L$ and $D=d_1d_2$ satisfy Assumption \ref{ass44}. 
By \cite[Proposition 5.15]{ACST}, a prime $p$ divides $(D, D_{L/K})$ if and only if $p$ divides $D$. 
Let $F=\Q(\sqrt{d_1d_2 \cdot t})E$ ($t>1$ is a squarefree integer prime to $d_1d_2$, $E$ is a number field of odd degree) and 
$$F_1=F(\sqrt{d_1}), \, F_2=F(\sqrt{d_2}), \, F_3=F(\sqrt{d_1d_2}). $$
Since $\Hom(A, \Z/2\Z)$ is of order $4$, these are all quadratic subfields of $FK$ over $F$. Suppose that $\Q(\sqrt{d_1d_2 t}) \cap L \neq \Q$. Then $\sqrt{d_1d_2 t} \in L$ so $L=\Q(\sqrt{d_1}, \sqrt{d_2}, \sqrt{d_1d_2 t})=\Q(\sqrt{d_1}, \sqrt{d_2}, \sqrt{t})$ and $\Gal(L/\Q) \cong (\Z/2\Z)^3$, which is a contradiction. Thus $\Q(\sqrt{d_1d_2 \cdot t}) \cap L = \Q$.

\begin{proposition} \label{prop49}
Let $F^{\alpha_i}=F_i$, $F^-=FK$, $F^+=FL$ and $\rho$ and $c_i=\alpha_i \cup \epsilon$ be chosen as above. Then,
\begin{equation*}
\begin{split}
CS_{c_1}([\rho]) &= \frac{1}{2} \Longleftrightarrow  \prod_{p \mid d_1}\left ( \frac{d_2t}{p} \right ) \cdot \prod_{p \mid d_2}\left ( \frac{d_1}{p} \right )=-1. \\
CS_{c_2}([\rho]) &= \frac{1}{2} \Longleftrightarrow  \prod_{p \mid d_1}\left ( \frac{d_2}{p} \right ) \cdot \prod_{p \mid d_2}\left ( \frac{d_1t}{p} \right )=-1. \\
CS_{c_3}([\rho]) &= \frac{1}{2} \Longleftrightarrow  \prod_{p \mid d_1d_2}\left ( \frac{t}{p} \right ) =-1. 
\end{split}
\end{equation*}
\end{proposition}

\noindent Two examples of $L, K$ for $A=\Z/2\Z \times \Z/2\Z$, $\Gamma=\mathcal{Q}_8$ in \cite[p. 21]{ACST} can be used in our situation. Let 
$$g(x) = x^8 - x^7 +98x^6 -105x^5 +3191x^4 +1665x^3 +44072x^2 +47933x +328171$$
(LMFDB \cite{LMF1}) be the irreducible polynomial over $\Q$ and $\beta$ be a root of $g(x)$. Also let
\begin{center}
$L=\Q(\beta)$, $K=\Q(\sqrt{5}, \sqrt{29})$ and $D=5 \cdot 29$.
\end{center}
Then $K \subset L$, $\Gal(L/\Q) \cong \mathcal{Q}_8$ and $d_L=3^4 \cdot 5^6 \cdot 29^6$ is odd. Let $F=\Q(\sqrt{145 \cdot t})E$, where $t>1$ is a squarefree integer prime to $145$ and $E$ is a number field of odd degree. 

\begin{corollary} \label{cor410}
Let $\rho$ and $c_i=\alpha_i \cup \epsilon$ be chosen as above. Then, 
\begin{equation*}
\begin{split}
CS_{c_1}([\rho]) &= \frac{1}{2} \Longleftrightarrow  \left ( \frac{t}{5} \right )=-1 \Longleftrightarrow  t \equiv \pm 2 \,\, (mod \,\, 5). \\
CS_{c_2}([\rho]) &= \frac{1}{2} \Longleftrightarrow  \left ( \frac{t}{29} \right )=-1. \\
CS_{c_3}([\rho]) &= \frac{1}{2} \Longleftrightarrow  \left ( \frac{t}{5} \right )=- \left ( \frac{t}{29} \right ). 
\end{split}
\end{equation*}
\end{corollary}

\noindent Let 
$$g(x)=x^8-x^7-34x^6+29x^5+361x^4-305x^3-1090x^2+1345x-395$$
(LMFDB \cite{LMF2}) be the irreducible polynomial over $\Q$ and $\beta$ be a root of $g(x)$. Also let
\begin{center}
$L=\Q(\beta)$, $K=\Q(\sqrt{5}, \sqrt{21})$, $D=5 \cdot 21$
\end{center}
and $F=\Q(\sqrt{105 \cdot t})E$, where $t>1$ is a squarefree integer prime to $105$ and $E$ is a number field of odd degree.

\begin{corollary} \label{cor411}
Let $\rho$ and $c_i=\alpha_i \cup \epsilon$ be chosen as above. Then, 
\begin{equation*}
\begin{split}
CS_{c_1}([\rho]) &= \frac{1}{2} \Longleftrightarrow  \left ( \frac{t}{5} \right )=-1 \Longleftrightarrow  t \equiv \pm 2 \,\, (mod \,\, 5). \\
CS_{c_2}([\rho]) &= \frac{1}{2} \Longleftrightarrow  \left ( \frac{t}{3} \right )=- \left ( \frac{t}{7} \right ). \\
CS_{c_3}([\rho]) &= \frac{1}{2} \Longleftrightarrow  \left ( \frac{t}{3} \right ) \cdot \left ( \frac{t}{5} \right ) \cdot \left ( \frac{t}{7} \right ) =-1. 
\end{split}
\end{equation*}
\end{corollary}

Now let $A=\Z/2\Z \times \Z/2\Z$ and $\Gamma = D_4$, the dihedral group of order $8$. Since $K$ is totally imaginary in the example in \cite{ACST}, we have to find a new example. Let
$$ g(x)=x^8-x^6-4x^4-16x^2+256 $$
(LMFDB \cite{LMF3}) which is irreducible over $\Q$ and let $\beta$ be a root of $g(x)$. Also let
\begin{center}
$L=\Q(\beta)$, $K=\Q(\sqrt{5}, \sqrt{29})$ and $D=5 \cdot 29$.
\end{center}
Then $K$, $L$ and $D$ satisfy Assumption \ref{ass44} and $D_{L/K}=5^2$. Let $F=\Q(\sqrt{145 \cdot t})E$ ($t>1$ is a squarefree integer prime to $145$ and $E$ is a number field of odd degree), $F_1=F(\sqrt{5})$, $F_2=F(\sqrt{29})$ and $F_3=F(\sqrt{145})$. We can check that $\Q(\sqrt{145 \cdot t}) \cap L = \Q$ as before. 

\begin{proposition} \label{prop412}
Let $F^{\alpha_i}=F_i$, $F^-=FK$, $F^+=FL$ and $\rho$ and $c_i=\alpha_i \cup \epsilon$ be chosen as above. Then, 
\begin{equation*}
\begin{split}
CS_{c_1}([\rho]) &= \frac{1}{2} \Longleftrightarrow  t \equiv \pm 2 \,\, (mod \,\, 5). \\
CS_{c_2}([\rho]) &= 0. \\
CS_{c_3}([\rho]) &= \frac{1}{2} \Longleftrightarrow  t \equiv \pm 2 \,\, (mod \,\, 5). 
\end{split}
\end{equation*}
\end{proposition}

\vspace{2em}

\noindent 

\textbf{Case III. $A=S_4$}:
Let $A=S_4$, the symmetric group of degree $4$. There is a unique surjective map $\alpha : A \rightarrow \Z/2\Z$ and three non-trivial central extensions $\Gamma_i$ ($1 \leq i \leq 3$) of $A$ by $\Z/2\Z$:
\begin{center}
$\Gamma_1=\GL(2, \mathbb{F}_3)$, $\Gamma_2$ is the transitive group `16T65' in \cite{GDB1} and $\Gamma_3=\SL(2, \Z/4\Z)$.
\end{center}

\noindent Let $\Gamma = \Gamma_1$ and suppose that $\Q  \subset \Q(\sqrt{D}) \subset K \subset L$ and $E$ satisfy Assumption \ref{ass44}. Let $F=\Q(\sqrt{Dt})E$, where $t>1$ is a squarefree integer prime to $D$. 
Since $\Gamma_1$ has a unique subgroup of order $24$ and $\Q(\sqrt{Dt}) \neq \Q(\sqrt{D})$, $\Q(\sqrt{Dt}) \cap L = \Q$. 

\begin{proposition} \label{prop413}
Let $F^{\alpha}=F(\sqrt{D})$, $F^-=FK$, $F^+=FL$ and $\rho$ and $c=\alpha \cup \epsilon$ as above. Then, 
\begin{equation*}
CS_c([\rho])=0.
\end{equation*}
\end{proposition}

\noindent Let
\begin{equation*}
\begin{split}
g_1(x) &= x^8-x^7-52x^6-66x^5+561x^4+1153x^3-851x^2-1857x+876 \\
g_2(x) &= x^4-x^3-5x^2+4x+3
\end{split}
\end{equation*}
(LMFDB \cite{LMF4}, \cite{LMF5}) be irreducible polynomials over $\Q$, $D=7537$ and $L$ and $K$ be the splitting fields of $g_1(x)$ and $g_2(x)$, respectively. Then $\Gal(L/\Q) \cong \Gamma$, $\Gal(K/\Q) \cong A$ and $K$ is totally real. 
Also $d_K=7537^{12}$ and $d_L=3^{24} \cdot 7537^{24}$ ($d_K$ and $d_L$ can be calculated from the data of \cite{LMF4}, \cite{LMF5} by using the conductor-discriminant formula), so $D_{L/K}=3^{24}$ and $D_{K/\Q(\sqrt{D})}=1$. 
Thus $\Q  \subset \Q(\sqrt{D}) \subset K \subset L$ and $E$ satisfy Assumption \ref{ass44} for a number field $E$ of odd degree such that $E \cap L = \Q$.
Let $F=\Q(\sqrt{7537 \cdot t})E$, where $t>1$ is a squarefree integer prime to $7537$.

\begin{corollary} \label{cor414}
Let $F^{\alpha}=F(\sqrt{7537})$, $F^-=FK$, $F^+=FL$ and $\rho$ and $c$ be chosen as above. Then, 
\begin{equation*}
CS_c([\rho])=0.
\end{equation*}
\end{corollary}

\noindent Now let $\Gamma = \Gamma_2$, 
\begin{equation*}
\begin{split}
& \scriptstyle g_1(x)=x^{16}+5x^{15}-790x^{14}-4654x^{13}+234254x^{12}+1612152x^{11}-33235504x^{10} \\ 
& \scriptstyle -263221982x^9+2331584048x^8+21321377994x^7-74566280958x^6-825209618478x^5 \\
& \scriptstyle +922238608476x^4+13790070608536x^3-6704968288135x^2-80794234036917x+87192014930816 
\end{split}
\end{equation*}
(Galoisdb \cite{GDB1}) and 
$$ g_2(x)=x^4-x^3-4x^2+x+2 $$
(LMFDB \cite{LMF6}) be irreducible polynomials over $\Q$, $D=2777$ and $L$ and $K$ be the splitting fields of $g_1(x)$ and $g_2(x)$, respectively. Then by \cite[Lemma 5.22]{ACST} and the fact that $K$ is totally real, $\Q \subset \Q(\sqrt{D}) \subset K \subset L$ and $E$ satisfy Assumption \ref{ass44} for a number field $E$ of odd degree such that $E \cap L = \Q$ and $\Q(\sqrt{2777 \cdot t}) \cap L = \Q$ for squarefree $t>1$ prime to $2777$. Let $F=\Q(\sqrt{2777 \cdot t})E$, where $t>1$ is a squarefree integer prime to $2777$.

\begin{proposition} \label{prop415}
Let $F^{\alpha}=F(\sqrt{2777})$, $F^-=FK$, $F^+=FL$ and $\rho$ and $c$ be chosen as above. Then, 
\begin{equation*}
CS_{c}([\rho]) = \frac{1}{2} \Longleftrightarrow  \left ( \frac{t}{2777} \right )=-1. 
\end{equation*}
\end{proposition}

\begin{remark} \label{rmk416}
Let $F=\Q(\sqrt{2777 \cdot 7537})E$ and $c_1, c_2 \in H^3(S_4, \Z/2\Z)$ are given by $c$ of Proposition \ref{prop413} and \ref{prop415}, respectively. Then,
\begin{center}
$CS_{c_1}([\rho])=0$ and $\displaystyle \left ( \frac{7537}{2777} \right )=-1 \, \Rightarrow \,  CS_{c_2}([\rho])=\frac{1}{2}$, 
\end{center}
so $CS_c([\rho])$ can be different when both the number field $F$ and the gauge group $A$ are same. 
\end{remark}

\subsection{Biquadratic fields with $n=2$} \label{Sub44}

\noindent In this subsection, we provide an analogue of Theorem \ref{thm46} for the compositum of a biquadratic field and a number field of odd degree. Unfortunately, $CS_c([\rho])$ is always zero in this case. 

\begin{assumption} \label{ass417}
$K \subset L$ and $E$ are number fields, $D_1, D_2 \neq 1$ are squarefree integers such that \\
(1) $\Gal(L/\Q) \cong \Gamma$ and $d_L$ is odd. \\
(2) $\Gal(K/\Q) \cong A$ and $K$ is totally real when both $D_1$ and $D_2$ are positive. \\
(3) $[E:\Q]$ is odd and $E \cap L = \Q$. \\
(4) $(D_1, D_2)=1$ and $D_1D_2$ divides $d_K$. \\
(5) $\Q(\sqrt{D_1}, \sqrt{D_2}) \subset K$ and $K/\Q(\sqrt{D_1}, \sqrt{D_2})$ is unramified at all primes.
\end{assumption}

\noindent Let $F=\Q(\sqrt{D_1t_1}, \sqrt{D_2t_2})E$, where $t_1, t_2>1$ are squarefree integers such that $D_1, D_2, t_1, t_2$ are pairwise coprime. Denote $F_0=\Q(\sqrt{D_1t_1}, \sqrt{D_2t_2})$ and suppose that $F_0 \cap L = \Q$. Then we can prove $F \cap L = \Q$ exactly as the proof of Proposotion \ref{prop45}.

\begin{proposition} \label{prop418}
Let $K$, $L$, $D_1$, $D_2$ and $E$ satisfy Assumption \ref{ass417} and $F$ is given as above (with the assumption $F_0 \cap L = \Q$). Then $F^-=KF$ and $F^+=LF$ satisfy Assumption \ref{ass41}. 
\end{proposition}

\noindent Now suppose that $F^{\alpha}=F(\sqrt{M}) \subset F^-$ for some $M \neq 1$ dividing $D_1D_2$. Denote $\Q_1 = \Q(\sqrt{M})$ and $\Q_2 = \Q(\sqrt{N})$ for $N=D_1D_2t_1t_2/M$.

\begin{theorem} \label{thm419}
Suppose we are in Assumption \ref{ass417} and choose $F \subset F^{\alpha} \subset F^- \subset F^+$ as above, $\rho$ and $c$ as Assumption \ref{ass41}. Then,
\begin{equation*}
CS_c([\rho]) =0.
\end{equation*}
\end{theorem}

\begin{proof}
We can prove that 
$$ S_f = \left \{  \mathfrak{p} \in \Spec(\mathcal{O}_F) :  \mathfrak{p} \mid D_{L/K}, \, \mathfrak{p} \nmid t_1t_2 \right \} $$
by using Abhyankar's lemma as in the proof of \cite[Theorem 5.13]{ACST}. By Theorem \ref{thm43}, $\displaystyle CS_c([\rho]) \equiv \frac{r}{2} \,\, (mod \,\, \Z)$ where $r$ is the number of primes in $S_f$ which are inert in $F^{\alpha}$. 
Let $\mathfrak{p} \in S_f$ be a prime above $p$ and $\mathfrak{p}_0=\mathfrak{p} \cap F_0$. For primes of $F$ above $p$, they are all inert or all split in $F^{\alpha}$. Thus if $p$ splits into two or four primes in $F_0$, these do not change the parity of $r$, so we do not need to consider them. \\
Now suppose that $\mathfrak{p}_0$ is the only prime of $F_0$ above $p$. If $p$ is unramified in $F_0/\Q$, then its inertia degree in $F_0/\Q$ is $4$ so 
$$ \left ( \frac{D_1t_1}{p} \right )=\left ( \frac{D_2t_2}{p} \right )=\left ( \frac{D_1D_2t_1t_2}{p} \right )=-1, $$
a contradiction. Suppose that both the ramification index and the inertia degree of $p$ in $F_0/\Q$ are $2$. One can choose
$$ F_1 \in \left \{ \Q(\sqrt{D_1t_1}, \sqrt{M}), \, \Q(\sqrt{D_1t_1}, \sqrt{N}), \, \Q(\sqrt{D_2t_2}, \sqrt{M}), \, \Q(\sqrt{D_2t_2}, \sqrt{N}) \right \} $$
which is unramified at $p$. Since $[F:F_0]=[E:\Q]$ is odd, $F^{\alpha}/F$ is inert at $\mathfrak{p}$ if and only if $F_0(\sqrt{M})/F_0$ is inert at $\mathfrak{p}_0$. In this case, the inertia degree of $p$ in $F_0(\sqrt{M})/\Q$ is $4$ and $F_0(\sqrt{M})/F_1$ is ramified at $p$ so the inertia degree of $p$ in $F_1/\Q$ is also $4$, which is also a contradiction. 
\end{proof}

\subsection{Non-vanishing examples with general $n$ and non-abelian gauge group} \label{Sub45}

In this subsection, $F$ is always totally imaginary. In \cite{BCG}, the following is proved. 

\begin{theorem} \label{thm4a}
(\cite[Theorem 1.2]{BCG}) Let $n>1$ be an integer and $c$ be a fixed generator of $H^3(\Z/n\Z, \Z/n\Z)$. Then, there are infinitely many totally imaginary number fields $F$ with a cyclic unramified Kummer extension $K/F$ of degree $n$ such that for the natural map $\rho : \pi \twoheadrightarrow \Gal(K/F) \xrightarrow{\simeq} \Z/n\Z$,
$$CS_c([\rho]) \neq 0. $$
\end{theorem}

\noindent For $n \geq 3$ and a non-cyclic gauge group $A$, nothing was known about the non-triviality of $CS_c$. Based on the theorem above, we prove for a large family of finite groups, there are infinitely many totally imaginary number fields $F$ such that $CS_c([\rho]) \neq 0$ for some $\rho$ and $c$. 

\begin{lemma} \label{lem4b}
Let $\varphi : B \rightarrow A$ be a homomorphism of finite groups, $\rho \in \Hom_{cont}(\pi, B)$ and $c \in H^3(A, \Z/n\Z)$. For $\rho'= \varphi \circ \rho : \pi \rightarrow A$ and $c'=\varphi^*(c) \in H^3(B, \Z/n\Z)$, 
$$CS_{c}([\rho'])= CS_{c'}([\rho]). $$
\end{lemma}

\begin{proof}
$(\rho')^*(c)= (\varphi \circ \rho)^*(c)=\rho^* \varphi^* (c) =\rho^*(c')$. 
\end{proof}

\noindent Let $A$ be a finite group and $\varphi : \Z/n\Z \rightarrow A$ be a group homomorphism. If there exists $c \in H^3(A, \Z/n\Z)$ such that $\varphi^*(c)$ is a generator of $H^3(\Z/n\Z, \Z/n\Z)$, then by Theorem \ref{thm4a} and Lemma \ref{lem4b}, 
$$
CS_c([\varphi \circ \rho]) = CS_{\varphi^*(c)}([\rho]) \neq 0
$$
for infinitely many number fields $F$ with some $\rho \in \Hom_{cont}(\pi, \Z/n\Z)$. 
For any finite group $G$, let
$$
\delta : H^1(G, \Z/n\Z) \rightarrow  H^2(G, \Z/n\Z)
$$ 
be the Bockstein homomorphism coming from the exact sequence $0 \rightarrow \Z/n\Z \overset{n}{\rightarrow} \Z/n^2\Z \rightarrow \Z/n\Z \rightarrow 0$. 
For the identity map $Id \in \Hom(\Z/n \Z, \Z/n \Z) = H^1(\Z/n \Z, \Z/n \Z)$, $Id \cup \delta(Id) \in H^3(\Z/n \Z, \Z/n \Z)$ is a generator of the group $H^3(\Z/n \Z, \Z/n \Z)$ (\cite[p. 5684]{ALN}). 
Suppose that there is $\psi \in \Hom(A, \Z/n \Z)$ such that $\psi \circ \varphi = Id$ and denote $c = \psi \cup \delta(\psi) \in H^3(A, \Z/n \Z)$. 
Since the following diagram
\[
\begin{tikzcd}
H^1(\Z/n \Z, \Z/n \Z) \arrow[r, "\delta"]
& H^2(\Z/n \Z, \Z/n \Z)  \\
H^1(A, \Z/n \Z) \arrow[r, "\delta"] \arrow[u, "\varphi^*"]
& H^2(A, \Z/n \Z) \arrow[u, "\varphi^*"]
\end{tikzcd}
\]
\noindent  commutes, 
$$\varphi^*(c)=\varphi^*(\psi) \cup \varphi^*(\delta(\psi)) = \varphi^*(\psi) \cup \delta (\varphi^*(\psi)) = Id \cup \delta(Id)$$
so $\varphi^*(c)$ is a generator of $H^3(\Z/n \Z, \Z/n \Z)$.

\begin{theorem} \label{thm4c}
Let $G$ be a finite group, $\alpha : \Z/n \Z \rightarrow \Aut(G)$ a group homomorphism and $A = G \rtimes_{\alpha} \Z/n \Z$ the semidirect product of $G$ by $\Z/n \Z$ with respect to $\alpha$. 
Then, there are infinitely many totally imaginary number fields $F$ such that
$$ CS_c([\rho]) \neq 0 $$
for some $\rho \in \Hom_{cont}(\pi, A)$ and $c \in H^3(A, \Z/n \Z)$.
\end{theorem}

\begin{proof}
By Theorem \ref{thm4a}, there are infinitely many totally imaginary number fields $F$ with with a cyclic unramified Kummer extension $K/F$ of degree $n$ such that for the natural map $\rho' : \pi \twoheadrightarrow \Gal(K/F) \xrightarrow{\simeq} \Z/n\Z$,
$$
CS_{Id \cup \delta(Id)}(\rho') \neq 0.
$$
By the discussion above, it is enough to show that there are group homomorphisms $\varphi : \Z/n\Z \rightarrow A$ and $\psi : A \rightarrow \Z/n\Z$ such that $\psi \circ \varphi = Id$. 
Since $A$ is a semidirect product of $G$ by $\Z/n \Z$, there are 
$\psi \in \Hom(A, \Z/n \Z)$ and $\varphi \in \Hom(\Z/n \Z, A)$
such that $\psi \circ \varphi = Id$. 
Then for $c = \psi \cup \delta(\psi) \in H^3(A, \Z/n \Z)$
and $\rho=\varphi \circ \rho' \in \Hom_{cont}(\pi, A)$, we have (Lemma \ref{lem4b})
\begin{equation*}
CS_c([\rho]) = CS_{\varphi^*(c)}([\rho'])=CS_{Id \cup \delta(Id)}([\rho'])  \neq 0.
\qedhere
\end{equation*}
\end{proof}

\begin{example} \label{ex4d}
For $d \geq 3$, the Heisenberg group 
$$ H_d(\Z/n\Z) := \left \{ \begin{bmatrix}
1 & \mathbf{a} & c\\ 
0 & I_{d-2} & \mathbf{b}\\ 
0 & 0 & 1
\end{bmatrix} \in M_d(\Z/n\Z) : \mathbf{a} \in M_{1, d-2}(\Z/n\Z), 
\mathbf{b} \in M_{d-2, 1}(\Z/n\Z), 
c \in \Z/n\Z \right \} $$
(with matrix multiplication) can be represented as a semidirect product 
$$
((\Z/n\Z)^{d-1} \rtimes (\Z/n\Z)^{d-3}) \rtimes (\Z/n\Z),
$$
so Theorem \ref{thm4c} can be applied to the case $A=H_d(\Z/n\Z)$. 
\end{example}

\begin{example} \label{ex4e}
Let $\mathbb{F}_{q}$ be the finite field of order $q=p^k \geq 3$ and $r$ be a positive integer. Then the general linear group $\GL(r, \mathbb{F}_q)$ can be represented as a semidirect product $\SL(r, \mathbb{F}_q) \rtimes (\mathbb{F}_q)^{\times}$ and $(\mathbb{F}_q)^{\times}$ is a cyclic group of order $q-1$, so Theorem \ref{thm4c} can be applied to $n=q-1$ and $A=\GL(r, \mathbb{F}_q)$. 
\end{example}

\begin{example} \label{ex4f}
Let $p$ be a prime. If $p$ is odd, there are two non-abelian groups of order $p^3$ : $H_3(\Z/p\Z)$ and $\Z/p^2 \Z \rtimes \Z/p\Z$. Theorem \ref{thm4c} can be applied to both groups. If $p=2$, then non-abelian groups of order $8$ are $D_4 \cong \Z/4\Z \rtimes \Z/2\Z$ and $\mathcal{Q}_8$. Since $\mathcal{Q}_8$ is not a semidirect product of its proper subgroups, Theorem \ref{thm4c} cannot be applied.
\end{example}

\begin{appendices}

\section{Completed \'etale cohomology of number fields} \label{appnA}

In this section, we review the method of Zink \cite{ZINK} and Conrad-Masullo \cite{CM} of generalizing the \'etale cohomology of $X=\Spec(\mathcal{O}_F)$ for totally imaginary number fields $F$, which was studied by Mazur \cite{MAZ}, to arbitrary number fields. And we give an alternative viewpoint of the arithmetic Chern-Simons theory by using this generalized \'etale cohomology. We closely follow the exposition of \cite{CM}.

\subsection{Completed small \'etale sites}

Let $\eta :\Spec(F) \rightarrow X$ be a generic point of $X$ and denote its image in $X$ also by $\eta$. For a scheme $Y$ and a group $G$, denote the category of abelian \'etale sheaves on $Y$ by $\mathbf{Ab}_Y$ and the category of $G$-modules by $\mathbf{Mod}_G$. For $\mathcal{F} \in \mathbf{Ab}_X$, $\eta^*(\mathcal{F}) \in \mathbf{Ab}_{\Spec(F)}$ and the functor $\mathbf{Ab}_{\Spec(F)} \rightarrow \mathbf{Mod}_{G_F}$ ($\mathcal{G} \mapsto \displaystyle\mathop{\lim_{\longrightarrow}}_{F \subseteq L \subseteq \overline{F}, \,  [L:F] < \infty} \mathcal{G}(\Spec (L))$) is an equivalence of categories (\cite[Proposition 5.7.8]{LF}) so $\eta^*(\mathcal{F})$ corresponds to a $G_F$-module. (Denote by $\mathcal{F}_{\eta}$.)

\noindent Let $X_{\infty} = \left \{ v_1, \cdots, v_r \right \}$ be the set of all real places of $F$ and $\overline{X} := X \cup X_{\infty}$. We endow $\overline{X}$ with the topology whose closed sets are finite subsets of $\overline{X} \setminus \left \{ \eta \right \}$. Denote an open set of $\overline{X}$ by $\widetilde{U}$ and define $\widetilde{U}_{\infty} := \widetilde{U} \cap X_{\infty}$ and $U := \widetilde{U} \cap X$. For an open subset $U \subset X$, define $\overline{U} := U \cup X_{\infty}$. 
For each real place $v$ of $F$, let $\overline{v}$ be a fixed extension of $v$ to $\overline{F}$ and denote the decomposition group and the inertia group of $\overline{v}$ by $D_v$ and $I_v$, respectively. \\
Since $v$ is a real place, $I_v=D_v=\Gal(\overline{F_v}/F_v) \cong \Z/2 \Z$. 
For $\mathcal{F} \in \mathbf{Ab}_U$, $I_v$ acts on $\mathcal{F}_{\eta} \in \mathbf{Mod}_{G_F}$ via a natural homomorphism $I_v \rightarrow G_F$. 
For $x \in X \setminus \left \{ \eta \right \}$, its decomposition group and inertia group are given by $D_x := \Gal(\overline{F_x}/F_x)$ and $I_x := \Gal(\overline{F_x}/F_x^{un})$, where $F_x^{un}$ is the maximal unramified extension of $F_x$ in $\overline{F_x}$.

\begin{definition}
An \textbf{abelian sheaf} $\widetilde{ \mathcal{F} }$ on a non-empty open subset $\widetilde{U} \subset \overline{X}$ is a triple 
$$ \widetilde{ \mathcal{F} } := (\big\{  \widetilde{ \mathcal{F} }_{\widetilde{v}} \big\} , \mathcal{F}, \big\{ \varphi_{\widetilde{v}} \big\}) $$
where $\widetilde{v}$ varies through $\widetilde{U}_{\infty}$, $\widetilde{ \mathcal{F} }_{\widetilde{v}}$ are abelian groups, $\mathcal{F}$ is an abelian \'etale sheaf on $U$ and $\varphi_{\widetilde{v}} : \widetilde{ \mathcal{F} }_{\widetilde{v}} \rightarrow \mathcal{F}_{\eta}^{I_{\widetilde{v}}}$ are group homomorphisms. We call $\widetilde{ \mathcal{F} }_{\widetilde{v}}$ the \textbf{stalk} of $\widetilde{ \mathcal{F} }$ at $\widetilde{v}$, and we also denote the stalk $\mathcal{F}_x$ at any $x \in U$ by $\widetilde{ \mathcal{F} }_x$. $\widetilde{ \mathcal{F} }$ is called \textbf{constructible} if $\mathcal{F}$ is constructible and $\widetilde{ \mathcal{F} }_{\widetilde{v}}$ is finite for each $\widetilde{v} \in \widetilde{U}_{\infty}$. 
\end{definition}

\begin{definition}
Let $\widetilde{ \mathcal{F} }=(\big\{ \widetilde{ \mathcal{F} }_{\widetilde{v}} \big\}, \mathcal{F}, \big\{ \varphi_{\widetilde{v}} \big\})$ and $\widetilde{ \mathcal{G} }=(\big\{ \widetilde{ \mathcal{G} }_{\widetilde{v}} \big\}, \mathcal{G}, \big\{ \psi_{\widetilde{v}} \big\})$ be abelian sheaves on $\widetilde{U}$. A \textbf{morphism} $\widetilde{f} : \widetilde{ \mathcal{F} } \rightarrow \widetilde{ \mathcal{G} }$ of abelian sheaves on $\widetilde{U}$ is a triple $\widetilde{f} := (( \widetilde{ f }_{\widetilde{v}} ), f)$ such that $f : \mathcal{F} \rightarrow \mathcal{G}$ is a morphism of sheaves and 
for each $\widetilde{v} \in \widetilde{U}_{\infty}$, 
$\widetilde{ f }_{\widetilde{v}} : \widetilde{ \mathcal{F} }_{\widetilde{v}} \rightarrow \widetilde{ \mathcal{G} }_{\widetilde{v}}$ is a group homomorphism and $f_{\eta} \circ \varphi_{\widetilde{v}} = \psi_{\widetilde{v}} \circ \widetilde{f}_{\widetilde{v}}$. The \textbf{kernel} and \textbf{cokernel} of $\widetilde{f}$ is defined by $\ker(\widetilde{f}) := (\big\{ \ker(\widetilde{ f }_{\widetilde{v}} )\big\}, \ker(f), \big\{ \varphi'_{\widetilde{v}}  \big\})$ and $\coker(\widetilde{f}) := (\big\{ \coker(\widetilde{ f }_{\widetilde{v}} )\big\}, \coker(f), \big\{ \psi'_{\widetilde{v}}  \big\})$ for canonical group homomorphisms $\varphi'_{\widetilde{v}}$ and $\psi'_{\widetilde{v}}$. 
\end{definition}

\noindent Denote the category of abelian sheaves on $\widetilde{U}$ by $\mathfrak{Ab}_{\widetilde{U}}$. To develop a cohomology theory on $\widetilde{U}$, we need to construct an \'etale site on which the category of abelian sheaves is $\mathfrak{Ab}_{\widetilde{U}}$. For any scheme $Y$, define $Y(\mathbb{R}) := \Hom(\Spec(\mathbb{R}), Y)$. 
A morphism of schemes $f : Y \rightarrow Y'$ induces a map $f^{\mathbb{R}} : Y(\mathbb{R}) \rightarrow Y'(\mathbb{R})$ and $X(\mathbb{R})$ can be identified with $X_{\infty}$. 
Let $\widetilde{\mathbf{Sch}}_X$ be the category whose objects are pairs $\widetilde{Y} = (Y, \widetilde{Y}_{\infty})$ for an \'etale $X$-scheme $Y$ and a subset $\widetilde{Y}_{\infty} \subset Y(\mathbb{R})$, and morphisms $\widetilde{g} : \widetilde{Y} \rightarrow \widetilde{Y'}$ are morphisms $g : Y \rightarrow Y'$ of $X$-schemes such that $g^{\mathbb{R}}(\widetilde{Y}_{\infty}) \subset \widetilde{Y'}_{\infty}$.

\noindent Now let $\mathbf{Et}_{\widetilde{U}}$ be the category whose objects are the morphisms $\widetilde{Y} \rightarrow \widetilde{U}$ in $\widetilde{\mathbf{Sch}}_X$ and morphisms are defined over $\widetilde{U}$ in the evident manner. 
A \textbf{covering} in $\mathbf{Et}_{\widetilde{U}}$ is a family of morphisms $\big\{ \widetilde{f_i} : \widetilde{Y_i} \rightarrow \widetilde{Y}  \big\}$ such that $\big\{ f_i : Y_i \rightarrow Y  \big\}$ is a covering in $\mathbf{Et}_{U}$ (so $\bigcup_{i \in I}f_i(Y_i)=Y$) and $\bigcup_{i \in I}f_i^{\mathbb{R}}(\widetilde{Y_i}_{\infty})=\widetilde{Y}_{\infty}$. 
These coverings define the \'etale topology on $\widetilde{U}$, so they define the small \'etale site $\widetilde{U}_{\acute{e}t}$. Denote the category of abelian sheaves on $\widetilde{U}_{\acute{e}t}$ by $\mathbf{Ab}_{\widetilde{U}}$. 

\begin{proposition} \label{propA3}
(\cite[Theorem 2.4.3]{CM}) The categories $\mathfrak{Ab}_{\widetilde{U}}$ and $\mathbf{Ab}_{\widetilde{U}}$ are equivalent. 
\end{proposition}

\noindent For an abelian group $A$, let $A_{\widetilde{U}}$ be the constant sheaf on $\widetilde{U}_{\acute{e}t}$ defined by $A$. 
From the proof of the above proposition in \cite{CM}, $A_{\widetilde{U}}$ corresponds to $(\big\{ A \big\}_{\widetilde{v}}, A_U, \big\{ id_A \big\}_{\widetilde{v}}) \in \mathfrak{Ab}_{\widetilde{U}}$. 
Denote this element also by $A_{\widetilde{U}}$. 
We may omit the subscript $\widetilde{U}$ of $A_{\widetilde{U}}$ when the context is clear. 
For $\widetilde{ \mathcal{F} } \in \mathbf{Ab}_{\widetilde{U}}$, 
$$\widetilde{ \mathcal{F} }(\widetilde{U}) \cong \Hom_{\mathbf{Ab}}(\Z, \widetilde{ \mathcal{F} }(\widetilde{U})) \cong \Hom_{\mathbf{Ab}_{\widetilde{U}}}(\Z_{\widetilde{U}}, \widetilde{ \mathcal{F} }) \cong \mathcal{F}(U) \times_{\mathcal{F}_{\eta}} \widetilde{ \mathcal{F} }_{\widetilde{v_1}} \times_{\mathcal{F}_{\eta}} \cdots \times_{\mathcal{F}_{\eta}} \widetilde{ \mathcal{F} }_{\widetilde{v_s}}$$ 
where $\widetilde{U}_{\infty}=\left \{ \widetilde{v_1}, \cdots, \widetilde{v_s} \right \}$, and the fiber product is taken with respect to the natural map $\mathcal{F}(U) \rightarrow \mathcal{F}_{\eta}$ and the specialization maps $\widetilde{\mathcal{F}}_{\widetilde{v_k}} \rightarrow \mathcal{F}_{\eta}^{I_{\widetilde{v_k}}} \xhookrightarrow{} \mathcal{F}_{\eta}$. 
The right side gives the definition of the global section functor on $\mathfrak{Ab}_{\widetilde{U}}$, namely, $\Gamma(\widetilde{U}, \cdot) : \mathfrak{Ab}_{\widetilde{U}} \rightarrow \mathbf{Ab}$ ($\widetilde{ \mathcal{F} } \mapsto \mathcal{F}(U) \times_{\mathcal{F}_{\eta}} \widetilde{ \mathcal{F} }_{\widetilde{v_1}} \times_{\mathcal{F}_{\eta}} \cdots \times_{\mathcal{F}_{\eta}} \widetilde{ \mathcal{F} }_{\widetilde{v_s}})$. 
From now on we identify two categories $\mathfrak{Ab}_{\widetilde{U}}$ and $\mathbf{Ab}_{\widetilde{U}}$, and use them without distinction.

\subsection{Cohomology theory}

Since $\Gamma(\widetilde{U}, \cdot)$ is left-exact, its right derived functor $H^p(\widetilde{U}, \widetilde{\mathcal{F}}) := R^p \Gamma(\widetilde{U}, \widetilde{\mathcal{F}})$ is well-defined for each $p \geq 0$. 
Let $\widetilde{V} \subset \widetilde{U}$ be non-empty open subsets of $\overline{X}$ and $\widetilde{S} = \widetilde{U} \setminus \widetilde{V}$. 
Then $\widetilde{S}=S_f \cup \widetilde{S}_{\infty}$ for $S_f=U \setminus V$ and $\widetilde{S}_{\infty}=\widetilde{U}_{\infty} \setminus \widetilde{V}_{\infty}$. 
Define $\Gamma_{\widetilde{S}}(\widetilde{U}, \widetilde{\mathcal{F}}) := \ker(\widetilde{\mathcal{F}}(\widetilde{U}) \rightarrow \ker(\widetilde{\mathcal{F}}(\widetilde{V}))$ and $H^p_{\widetilde{S}}(\widetilde{U}, \widetilde{\mathcal{F}}) := R^p \Gamma_{\widetilde{S}}(\widetilde{U}, \widetilde{\mathcal{F}})$. 
Denote $\Gamma_{\widetilde{S}}(\widetilde{U}, \widetilde{\mathcal{F}})$ by $\Gamma_{x}(\widetilde{U}, \widetilde{\mathcal{F}})$ and $H^p_{\widetilde{S}}(\widetilde{U}, \widetilde{\mathcal{F}})$ by $H^p_{x}(\widetilde{U}, \widetilde{\mathcal{F}})$ when $\widetilde{S} = \left \{ x \right \}$. 
Then for $\widetilde{\mathcal{F}} \in \mathbf{Ab}_{\widetilde{U}}$, the following local cohomology sequence is exact (\cite[p. 21]{CM}).
$$ \cdots \rightarrow H^{p-1}(\widetilde{U}, \widetilde{\mathcal{F}}) \rightarrow  H^{p-1}(\widetilde{V}, \widetilde{\mathcal{F}} \mid_{\widetilde{V}}) \rightarrow \bigoplus_{x \in \widetilde{S}} H_x^p(\widetilde{U}, \widetilde{\mathcal{F}}) \rightarrow  H^p(\widetilde{U}, \widetilde{\mathcal{F}}) \rightarrow  H^p(\widetilde{V}, \widetilde{\mathcal{F}} \mid_{\widetilde{V}}) \rightarrow  \cdots $$

\begin{lemma} \label{lemA4}
(\cite[Lemma 3.2.3]{CM}) For $\widetilde{ \mathcal{F} } = (\big\{  \widetilde{ \mathcal{F} }_{\widetilde{v}} \big\} , \mathcal{F}, \big\{ \varphi_{\widetilde{v}} \big\}) \in \mathbf{Ab}_{\widetilde{U}}$ and $\widetilde{v} \in \widetilde{U}_{\infty}$, 
$$ H^p_{\widetilde{v}}(\widetilde{U}, \widetilde{ \mathcal{F} }) = 
\left\{\begin{matrix}
\ker (\varphi_{\widetilde{v}})  & (p=0)   \\ 
\coker (\varphi_{\widetilde{v}}) & (p=1) \\ 
H^{p-1}(I_{\widetilde{v}}, \mathcal{F}_{\eta}) & (p \geq 2)
\end{matrix}\right. .$$
\end{lemma}

\noindent Let $U \subset X$ be a non-empty open subscheme. For each $v \in X_{\infty}=\overline{U}_{\infty}$, denote the algebraic closure of $F$ in its completion $F_v \cong \mathbb{R}$ by $F_v^{alg}$. Its positive part, which corresponds to $\mathbb{R}_{>0}$ by the isomorphism $F_v \simeq \mathbb{R}$ is denoted by $F_v^{alg, +}$. Denote the inclusion of groups $F_v^{alg, +}  \hookrightarrow  F_v^{alg, \times}$ by $\iota_v$. From the formula
$$\displaystyle \mathbf{G}_{m, U, \eta}= \mathop{\lim_{\longrightarrow}}_{F \subseteq L \subseteq \overline{F}, \,  [L:F] < \infty} \mathbf{G}_{m, U}(\Spec(L)) = \mathop{\lim_{\longrightarrow}}_{F \subseteq L \subseteq \overline{F}, \,  [L:F] < \infty} L^{\times} = \overline{F}^{\times},$$
$\mathbf{G}_{m, U, \eta}^{I_v} = F_v^{alg, \times}$ and the following definition is well-defined. For the motivation of the definition below, see Example \ref{exA8}.

\begin{definition}
The sheaf $\mathbf{G}_{m, \overline{U}} \in \mathfrak{Ab}_{\overline{U}}=\mathbf{Ab}_{\overline{U}}$ is defined by the triple
$(\big\{ F_v^{alg, +} \big\}_{v \in X_{\infty}}, \mathbf{G}_{m, U}, \big\{ \iota_v \big\}_{v \in X_{\infty}})$. 
\end{definition}

\noindent Define the sheaf $\mu_{n, \overline{U}} \in \mathbf{Ab}_{\overline{U}}$ by $\mu_{n, \overline{U}} :=\ker(\mathbf{G}_{m, \overline{U}} \overset{n}{\rightarrow} \mathbf{G}_{m, \overline{U}})$. Since there is an injection $F_v^{alg, +} \rightarrow \mathbb{R}_{>0}$, $\ker(F_v^{alg, +} \overset{n}{\rightarrow} F_v^{alg, +})=0$ and $\mu_{n, \overline{U}} = ( \big\{ 0 \big\}_{v \in X_{\infty}}, \mu_{n, U}, \big\{ 0 \rightarrow \mu_n(F_v^{alg, +}) \big\}_{v \in X_{\infty}})$. We may omit the subscript $\overline{U}$ of $\mathbf{G}_{m, \overline{U}}$ and $\mu_{n, \overline{U}}$ when the context is clear. 
Even though the cohomology of $\mathbf{G}_{m, \overline{U}}$ for an arbitrary non-empty open subset $U$ of $X$ is known, we only state the result for the case $U=X$ because this is the only case that we use later. Let $\mathcal{O}_F^{\times, +}$ be the group of totally positive units of $F$, and $Cl^+(F)$ be the narrow class group of $F$.

\begin{proposition} \label{propA6}
(\cite[Proposition 2.5.9]{ZINK}) $$H^0(\overline{X}, \mathbf{G}_{m, \overline{X}})=\mathcal{O}_F^{\times, +}, \, H^1(\overline{X}, \mathbf{G}_{m, \overline{X}})=Cl^+(F), \, H^2(\overline{X}, \mathbf{G}_{m, \overline{X}})=0, H^3(\overline{X}, \mathbf{G}_{m, \overline{X}})=\Q/\Z.$$
\end{proposition}

\noindent The cohomology of $\mu_{n, \overline{X}}$ and $\Z/n\Z_{\overline{X}}$ will be given in the next subsection. To compute these, we should use Artin-Verdier duality for an arbitrary number field.

\subsection{Artin-Verdier duality}

Let $\mathcal{F} \in \mathbf{Ab}_{U}$ be a non-empty open subset $U$ of $X$. For each $v \in X_{\infty}$, $\mathcal{F}_{\eta}$ is an $I_v$-module and the norm map $N_v$ on $\mathcal{F}_{\eta}$ induces a map $N_v : H_0(I_v, \mathcal{F}_{\eta}) \rightarrow H^0(I_v, \mathcal{F}_{\eta})$. 

\begin{definition}
The \textbf{modified sheaf} of $\mathcal{F} \in \mathbf{Ab}_{U}$ is defined by 
$$ \widehat{\mathcal{F}} := (  \big\{  H_0(I_v, \mathcal{F}_{\eta}) \big\}_{v \in X_{\infty}} , \mathcal{F}, \big\{ H_0(I_v, \mathcal{F}_{\eta}) \xrightarrow{N_v} H^0(I_v, \mathcal{F}_{\eta}) \big\}_{v \in X_{\infty}}  ) \in \mathbf{Ab}_{\overline{U}}. $$
\end{definition}

\begin{example} \label{exA8}
(1) Since $\displaystyle \mathbf{G}_{m, U, \eta}=\overline{F}^{\times}$ and the image of $N_v : \mathbf{G}_{m, U, \eta} \rightarrow \mathbf{G}_{m, U, \eta}$ is $F_v^{alg, +}$, 
$$H_0(I_v, \mathbf{G}_{m, U, \eta}) \overset{N_v}{\underset{\simeq}{\longrightarrow}} F_v^{alg, +}$$
and $N_v : H_0(I_v, \mathbf{G}_{m, U, \eta}) \rightarrow H^0(I_v, \mathbf{G}_{m, U, \eta})$ is identified with the inclusion map $F_v^{alg, +} \xhookrightarrow{} F_v^{alg, \times}$. Thus $\widehat{\mathbf{G}_{m, U}}=\mathbf{G}_{m, \overline{U}}$.

\noindent (2) Let $A$ be an abelian group. It is easy to show that $ \widehat{A_U} = ( \big\{ A \big\} , A_U, \big\{ A\overset{2}{\rightarrow} A \big\}) \in \mathbf{Ab}_{\overline{U}}$, which is not equal to $A_{\overline{U}}$ in general. However, there is a natural morphism $\widetilde{f} =( \big\{ A\overset{2}{\rightarrow} A \big\} , id_{A_U}) : \widehat{A_U} \rightarrow  A_{\overline{U}}$ and it induces an isomorphism $H^p(\overline{U}, \widehat{A_U}) \rightarrow  H^p(\overline{U}, A_{\overline{U}})$ for each $p \geq 1$ by the following lemma. 
\end{example}

\begin{lemma} \label{lemA9}
Let $\widetilde{ \mathcal{F} }=(\big\{ \widetilde{ \mathcal{F} }_{\widetilde{v}} \big\}, \mathcal{F}, \big\{ \varphi_{\widetilde{v}} \big\})$ and $\widetilde{ \mathcal{G} }=(\big\{ \widetilde{ \mathcal{G} }_{\widetilde{v}} \big\}, \mathcal{G}, \big\{ \psi_{\widetilde{v}} \big\})$ be abelian sheaves on $\overline{U}$ and $\widetilde{f} : \widetilde{\mathcal{F}} \rightarrow \widetilde{\mathcal{G}}$ be a morphism in $\mathbf{Ab}_{\overline{U}}$ whose kernel and cokernel are supported at real points. (Equivalently, $f : \mathcal{F} \rightarrow \mathcal{G}$ is an isomorphism in $\mathbf{Ab}_{U}$.) Then for each $p \geq 1$, the map $H^p(\overline{U}, \widetilde{\mathcal{F}}) \rightarrow  H^p(\overline{U}, \widetilde{\mathcal{G}})$ induced by $\widetilde{f}$ is an isomorphism. 
\end{lemma}

\begin{definition}
Let $\mathcal{F} \in \mathbf{Ab}_{U}$, $j : U \rightarrow X$ be an open immersion, $X_{\infty} = \left \{ v_1, \cdots, v_r \right \}$, $p \in \Z$ and $H^p_T$ denote the $p$-th Tate cohomology. \\
(1) The $p$-th \textbf{modified \'etale cohomology groups} of $\mathcal{F}$ are 
$$\widehat{H}^p(U, \mathcal{F}) :=
\left\{\begin{matrix}
H^p(\overline{U}, \widehat{\mathcal{F}})  & (p \geq 0)\\ 
\bigoplus ^r _{k=1} H^{p-1}_T (I_{v_k}, \mathcal{F}_{\eta})  & (p<0)
\end{matrix}\right. .$$
(2) The $p$-th \textbf{compactly supported cohomology groups} of $\mathcal{F}$ are $H^p_c(U, \mathcal{F}) := \widehat{H}^p(X, j_! \mathcal{F})$.
\end{definition}

\begin{example} \label{exA11}
(\cite[Example 5.4.2]{CM}) A natural morphism $\widehat{j_! \mathbf{G}_{m, U}} \rightarrow \mathbf{G}_{m, \overline{X}}$ in $\mathbf{Ab}_{\overline{X}}$ induces an isomorphism $H^3_c(U, \mathbf{G}_{m, U})=H^3(\overline{X}, \widehat{j_! \mathbf{G}_{m, U}}) \xrightarrow{\simeq} H^3(\overline{X}, \mathbf{G}_{m, \overline{X}}) \cong \Q/\Z$. 
\end{example}

\begin{theorem} \label{thmA12}
(\cite[Theorem 5.4.4]{CM}; Artin-Verdier duality) Let $\mathcal{F}$ be a constructible abelian \'etale sheaf on a dense open subset $U$ of $X$. The Yoneda pairing 
\begin{equation*}
H^p_c(U,\mathcal{F}) \times Ext_U^{3-p}(\mathcal{F}, \mathbf{G}_{m,U}) \rightarrow  H^3_c(U, \mathbf{G}_{m,U}) \cong \Q/\Z
\end{equation*}
is a perfect pairing of finite abelian groups for all $p \in \Z$. ($Ext^i_U$ is defined to be $0$ for $i<0$, so $H^p_c(U, \mathcal{F})=0$ if $p>3$.)
\end{theorem}

\noindent By Theorem \ref{thm23} and Theorem \ref{thmA12}, two different definitions of compactly supported \'etale cohomology in Section \ref{Sec2} and Appendix \ref{appnA} are naturally identified. For the isomorphism 
$$
\inv : H^3(\overline{X}, \Z/n\Z) \xleftarrow{\simeq} H^3_c(X, \Z/n\Z) \xrightarrow{\simeq} \mu_n(F)^D
$$
and the map 
$$
j^3 : H^3(\pi_{un}, \Z/n\Z) \xrightarrow{j^3_{un}} H^3_c(X, \Z/n\Z) \xrightarrow{\simeq} H^3(\overline{X}, \Z/n\Z),
$$
the arithmetic Chern-Simons action can be expressed as $CS_c([\rho]) = \inv(j^3(\rho^*(c)))$. 
\end{appendices}

\section*{Acknowledgments}
Jungin Lee was supported by a KIAS Individual Grant (SP079601) via the Center for Mathematical Challenges at Korea Institute for Advanced Study.
Jeehoon Park was supported by Samsung Science \& Technology Foundation SSTF-BA1502, 
the National Research Foundation of Korea (NRF-2021R1A2C1006696) and the National Research Foundation of Korea (NRF) grant funded by the Korea government (MSIT) (No.2020R1A5A1016126). 
 The second author thanks Hwajong Yoo for providing useful comments on the draft.


\vspace{3mm}

\footnotesize{
\textsc{Jungin Lee: Center for Mathematical Challenges, Korea Institute for Advanced Study, Seoul 02455, Republic of Korea} 

\textit{E-mail address}: \changeurlcolor{black}\href{mailto:jilee.math@gmail.com}{jilee.math@gmail.com} 

\vspace{3mm}

\textsc{Jeehoon Park:  QSMS, Seoul National University, 1 Gwanak-ro, Gwanak-gu, Seoul 08826, Republic of Korea }

\textit{E-mail address}: \changeurlcolor{black}\href{mailto:jpark.math@gmail.com}{jpark.math@gmail.com}
}

\end{document}